\documentclass[11pt]{article} 

\title{Extremal behavior of large cells in the\\ Poisson hyperplane mosaic}

\author{Moritz Otto\thanks{Department of Mathematics, 
		Aarhus University, 
		Ny Munkegade 118, 8000 Aarhus C, Denmark, email: otto@math.au.dk}
}

\usepackage{amssymb}
\usepackage{amsfonts}
\usepackage{amsmath}
\usepackage{amsthm}
\usepackage{bm}
\usepackage{graphicx}
\usepackage{multicol}
\usepackage{multirow}
\usepackage{subcaption}
\usepackage[usenames, dvipsnames]{xcolor}
\usepackage{verbatim}
\usepackage{dsfont}
\usepackage{color}
\usepackage[numbers]{natbib}
\usepackage{relsize}
\usepackage{lmodern}
\usepackage{url}
\usepackage{MnSymbol}
\usepackage{dsfont}
\usepackage{tikz}
\usepackage{float}
\usepackage{multicol}
\usepackage{vwcol}
\usepackage{mathtools}
\usepackage[T1]{fontenc}
\usepackage[utf8]{inputenc}
\usepackage{enumitem}   

\usepackage{accents}

\newcommand{\red}{\color{red}}

\theoremstyle{plain}
\newtheorem{theorem}{Theorem}
\newtheorem{lemma}[theorem]{Lemma}
\newtheorem{remark}[theorem]{Remark}
\newtheorem{proposition}[theorem]{Proposition}
\newtheorem{definition}[theorem]{Definition}
\newtheorem{corollary}[theorem]{Corollary}
\newtheorem{example}[theorem]{Example}
\newtheorem{fig}[theorem]{Figure}

\numberwithin{equation}{section}
\usepackage{caption}
\newcommand{\bthe}{\begin{theorem}}
	\newcommand{\ethe}{\end{theorem}}

\newcommand{\ben}{\begin{enumerate}}
	\newcommand{\een}{\end{enumerate}}

\newcommand{\bit}{\begin{itemize}}
	\newcommand{\eit}{\end{itemize}}

\newcommand{\beq}{\begin{equation}}
\newcommand{\eeq}{\end{equation}}

\newcommand{\ble}{\begin{lemma}}
	\newcommand{\ele}{\end{lemma}}

\newcommand{\bde}{\begin{definition}\rm}
	\newcommand{\ede}{\halmos\end{definition}}

\newcommand{\bco}{\begin{corollary}}
	\newcommand{\eco}{\end{corollary}}

\newcommand{\bpr}{\begin{proposition}}
	\newcommand{\epr}{\end{proposition}}

\newcommand{\brem}{\begin{remark}\rm}
	\newcommand{\erem}{\halmos\end{remark}}

\newcommand{\bproof}{\begin{proof}[Proof]}
	\newcommand{\eproof}{\end{proof}}

\newcommand{\bexam}{\begin{example}\rm}
	\newcommand{\eexam}{\halmos\end{example}}

\newcommand{\bexamwh}{\begin{example}\rm}
	\newcommand{\eexamwh}{\end{example}}

\newcommand{\bfi}{\begin{fig}}
	\newcommand{\efi}{\end{fig}}

\newcommand{\btab}{\begin{tab}}
	\newcommand{\etab}{\end{tab}}

\newcommand{\beao}{\begin{eqnarray*}}
	\newcommand{\eeao}{\end{eqnarray*}\noindent}

\newcommand{\beam}{\begin{eqnarray}}
\newcommand{\eeam}{\end{eqnarray}\noindent}

\newcommand{\ovr}{\begin{array}}
	\newcommand{\barr}{\begin{array}}
		\newcommand{\earr}{\end{array}}
	
	\newcommand{\bdis}{\begin{displaymath}}
	\newcommand{\edis}{\end{displaymath}\noindent}

	\def\N{{\mathbb N}}
	
	\def\P{{\mathbb P}}
	\def\E{{\mathbb E}}
	\def\R{{\mathbb R}}
	\def\K{{\mathcal{K}}}
	
	\def\F{{\mathcal F}}
	
	\def\H{{\mathbb H}}
	
	\def\phi{\varphi}
	\def\B{{\mathcal B}}
	\def\cals_+{{\cals_+}}

	\def\cals{{\mathcal{S}}}

	\newcommand{\norm}[1]{\left\lVert#1\right\rVert}

	\def\1{\mathbf{1}}

	\newcommand{\tto}{{t\to\infty}}

	\newcommand{\al}{\alpha}

	\newcommand{\eps}{\varepsilon}

	\newcommand{\vt}{\vartheta}

	\newcommand{\halmos}{\quad\hfill\mbox{$\Box$}}

\definecolor{plum}{cmyk}{0.50,1,0,0}
\definecolor{TealBlue}{cmyk}{0.86,0,0.34,0.02}
\definecolor{OliveGreen}{cmyk}{0.64,0,0.95,0.40}

\newcommand{\Nadine}{\textcolor{red}}

\begin{document}


\maketitle

\begin{abstract} 
	We study the asymptotic behavior of a size-marked point process of centers of large cells in a stationary and isotropic Poisson hyperplane mosaic in dimension $d \ge 2$. The sizes of the cells are measured by their inradius or their $k$th intrinsic volume ($k \ge 2$), for example. We prove a Poisson limit theorem for this process in Kantorovich-Rubinstein distance and thereby generalize a result in Chenavier and Hemsley (2016) in various directions. Our proof is based on a general Poisson process approximation result that extends a theorem in Bobrowski, Schulte and Yogeshwaran (2021).
\end{abstract}



\section{Introduction}\label{s1}
	Random mosaics are important objects in both theory and practice of modern probability. They are used in a wide range of applications to model random spatial phenomena. Besides the Voronoi mosaic (and its dual, the Delaunay mosaic), the hyperplane mosaic is an important model class. In this article we study stationary and isotropic Poisson hyperplane mosaics. These are random mosaics where the generating process is a stationary and isotropic Poisson hyperplane process.

Different aspects of this mosaic have been investigated in the literature. For example, in \cite{hug2007asymptotic} it is shown that the shape of the zero cell (it is the cell containing the origin $o \in \R^d$) is with high probability close to the shape of a ball if its $k$th intrinsic volume (for some $k \ge 2$) is large. This result is the answer to the Kendall problem (formulated for the zero cell of a Poisson hyperplane mosaic) and is adapted to the typical cell of the mosaic in \cite{hug2007typical}. In \cite{chenavier2016extremes} cells with large (and small) inradius and with center in a window are considered. For dimension $d=2$ it is shown that the limit distribution of the largest and smallest order statistics for the inradii converge to a Poisson distribution when the size of the window goes to infinity.

In the present paper we considerably extend this result in various directions. We consider marked point processes of centers of large (w.r.t. the inradius or an intrinsic volume, e.g. volume, surface area) cells in a compact window in a stationary, isotropic Poisson hyperplane mosaic, where the mark is a transformation of the cell size. We study the asymptotics of a scaling of this process as the volume of the window tends to infinity. Using Stein's method (via a coupling of the marked process with a Palm version of itself), we prove convergence in the Kantorovich-Rubinstein distance to a marked Poisson process. To this end, we give an extension of \cite[Theorem 4.1]{BSY21}. In the proofs of our main Theorems 1 and 2 we face two kinds of obstacles. The first one concerns long range dependencies and comes from the fact that two cells that are arbitrarily far apart from each other can share joint facet hyperplanes. The second one deals with local dependencies in the mosaic and concerns the sizes of the clusters in which large cells appear. The typical cluster size heavily depends on the shape of large cells. For the size functions that we consider (where all extremal bodies are balls) we use the answer to Kendall's problem and obtain clusters of size one. 

Our article is structured as follows. In Section 2 we give precise definitions of the objects that we study and present our main results. In Section 3 we prove two important auxiliary lemmas that give a bound on the number of hyperplanes in a stationary, isotropic Poisson hyperplane process that hit two disjoint balls. Moreover, we present a spherical Blaschke-Petkantschin formula that generalizes \cite[Theorem 7.3.2]{schneider2008stochastic} and might be of independent interest. In Section 4 we give an extension of \cite[Theorem 6.1]{BSY21}. In Section 5 we give the proof of Theorem \ref{thmr} (where we consider the process of centers of cells with a large inradius). The remaining sections are devoted to the proof of Theorem \ref{thmsigma} (which is a generalization of the proof of Theorem \ref{thmr} to the process of centers of cells that are large with respect to a more general size function). In Section 6 we study the distribution of a size functional of the typical cell in the Poisson hyperplane mosaic in more detail. In Section 7 we construct stopping sets for large cells in the mosaic and demonstrate how they can be used to derive a decorrelation inequality. Finally, we complete the proof of Theorem \ref{thmsigma} in Section 8.

\section{Preliminaries and Main Results}
We use the notation from \cite{hug2007typical} and \cite{BSY21}. For a locally compact second countable Hausdorff (lcscH) space $(\mathbb{X},\mathcal{X})$ we write $\mathbf{N}_{\mathbb{X}}$ for the space of all $\sigma$-finite counting measures on $\mathbb{X}$ and $\widehat{\mathbf{N}}_{\mathbb{X}}$ for the space of all finite counting measures on $\mathbb{X}$. We equip $\mathbf{N}_{\mathbb{X}}$ and $\widehat{\mathbf{N}}_{\mathbb{X}}$ with their corresponding $\sigma$-algebras $\mathcal{N}_{\mathbb{X}}$ and $\widehat{\mathcal{N}}_{\mathbb{X}}$ that are induced by the mappings $\omega \mapsto \omega(B)$ for all $B \in \mathcal{X}$. These are the Borel-$\sigma$-algebras with respect to the Fell topologies on $\mathbf{N}_{\mathbb{X}}$ and $\widehat{\mathbf{N}}_{\mathbb{X}}$, respectively. For  $x \in \mathbb{X}$ let $\delta_x \in \mathbf{N}_{\mathbb{X}}$ be the Dirac measure in $x$, i.e. $\delta_x(B)=\mathbf 1_B(x)$ for $B \in \mathcal{X}$. More generally, for $\bm{x}=(x_1,\dots,x_m) \in \mathbb{X}^m$ we write $\delta_{\bm{x}}=\delta_{x_1}+\cdots+\delta_{x_m}$. 

The real Euclidean vector space $\R^d$ is equipped with the standard scalar product $\langle \cdot,\cdot \rangle$ and induced norm $\|\cdot\|$. We write $\kappa_d$ for the volume of the closed unit ball $B^d:=\{x \in \R^d:\,\|x\|\le 1\}$, $\omega_d=d \kappa_d$ for the surface area of the unit sphere $S^{d-1}:=\partial B^d=\{x \in \R^d:\,\|x\|=1\}$ and $B(z,r):=z+rB^d=\{x \in \R^d:\,\|x-z\| \le r\}$ for the closed ball with radius $r>0$ around $z \in \R^d$. The $\ell$-dimensional ($\ell \in \{0,\dots,d\}$) Lebesgue measure on an $\ell$-dimensional affine subspace of $\R^d$ is denoted by $\lambda_\ell$ and $\sigma_k$ is the normalized Lebesgue measure on a $k$-dimensional ($k \in \{0,\dots,d-1\}$) great subsphere of $S^{d-1}$. Moreover, for a linear subspace $L\subset \mathbb{R}^d$ we let $S_L:=\{u \in L:\|u\|=1\}$.

The lcscH space (with the standard topology) of hyperplanes in $\R^d$ is denoted by $\H$. Every element $H \in \mathbb{H}$ can by represented as $H=H(u,r)=\{x \in \R^d:\,\langle x,u \rangle =r\}$ for some $u \in S^{d-1}$ and $r \in \R$. For $H \in \H$ and $z \in \R^d\setminus H$ we write $H_z^-$ for the closed halfspace bounded by $H$ that contains $z$ and we denote $H^-:=H_o^-$ for the closed halfspace bounded by $H$ (not passing through the origin $o \in \R^d$) that contains $o$. 

In this article all random objects are defined on some fixed probability space $(\Omega, \mathbf{A},\P)$ and $\eta$ is a stationary and isotropic Poisson hyperplane process in $\R^d$. This is a random element in the space $\mathbf{N}_\H$. As usual we write $\eta$ for the simple (induced) counting measure and its support. The distribution of $\eta$ is invariant under rotations and translations and its intensity measure $\E \eta$ is of the form $\E \eta=\gamma \mu_{d-1}$, where $\gamma>0$ is the intensity of $\eta$ and $\mu_{d-1}$ is the motion invariant measure on $\H$ given by
\begin{align}
	\mu_{d-1}(\cdot)=2 \int_{S^{d-1}} \int_0^\infty \mathbf{1}\{H(u,r)\in \cdot \} \,\mathrm{d}r\,\sigma_{d-1}(\mathrm{d}u). \label{repmu}
\end{align}

The space $\mathcal{K}^d$ of convex bodies (non-empty, compact, convex subsets of $\R^d$) is endowed with the Hausdorff metric. For $K \in \mathcal{K}^d$ we write $\H_K:=\{H \in \H:\,H \cap K \neq \emptyset\}$ and $\H^K:=\{H \in \H:\,H \cap K = \emptyset\}$. The random number $\eta(\H_K)$ of hyperplanes passing through $K$ follows a Poisson distribution with parameter $\gamma \Phi(K)$ with
\begin{align}
	\Phi(K):=2\int_{S^{d-1}} h(K,u)\,\sigma_{d-1}(\mathrm{d}u), \label{defphi}
\end{align}
where $h(K,\cdot)$ is the support function of $K$.

In this article we study point processes related to large cells in Poisson hyperplane mosaics. Let $\omega \in \mathbf{N}_\mathbb{H}$ be a locally finite system of hyperplanes in $\mathbb{R}^d$ in general position. This means that every $k$-dimensional plane of $\mathbb{R}^d$ is contained in at most $d-k$ hyperplanes of $\omega$. The closures of the connected components of the complement of $\bigcup_{H \in \eta} H$ in $\R^d$ are called {\em cells} of the {\em mosaic} generateed by $\omega$. For a tuple $\bm{H}=(H_1,\dots,H_{d+1})$ of $d+1$ pairwise distinct hyperplanes in $\omega$ let $B(\bm{H})$ be the closed inball of the unique simplex $\Delta(\bm{H})$ for which $H_1,\dots,H_{d+1}$ are the facet hyperplanes. Moreover, let $z(\bm{H})$ be the center and $r(\bm{H})$ be the radius of $B(\bm{H})$. Since $\omega$ is in general position, the inballs of the cells are unique and every inball is touched by precisely $d+1$ hyperplanes in $\omega$. For every cell this allows us to find unique (up to permutations) facet hyperplanes $H_1,\dots,H_{d+1} \in \omega$. We write
\begin{align}
	C(\bm{H},\omega):= \Delta(\bm{H})\cap \bigcap_{H \in \omega \cap \H^{B(\bm{H})}} H_{z(\bm{H})}^-. \label{defC}
\end{align}
It is easy to see that almost surely any $d+1$ hyperplanes of $\eta$ are in general position. The resulting Poisson hyperplane mosaic is stationary and isotropic and has intensity 
\begin{align}
	\gamma^{(d)}:=\frac{(2 \gamma)^d}{d+1}\int_{\mathsf{P}}  \Delta_d(\bm{u})  \sigma_{d-1}^{d+1}(\mathrm{d}\bm{u}) \label{defgammad}
\end{align}
(see \cite[Section 10.3]{schneider2008stochastic}), where $\mathsf{P} \subset (S^{d-1})^{d+1}$ is the set of all $(d+1)$-tuples of unit vectors not lying in a closed hemisphere and $\Delta_d(\bm{u})$ is the $d$-dimensional volume of the convex hull of $u_1,\dots,u_{d+1}$. 

For $n \ge 1$ and $\omega \in \mathbf{N}_{\H}$ we consider the point process 
\begin{align}
	\zeta_{n}[\omega]=\frac{1}{(d+1)!} \sum_{\bm{H}\in \omega^{(d+1)}}&  \mathbf{1}\{(\omega-\delta_{\bm{H}})\cap \H_{B(\bm{H})}= \emptyset \}\,\delta_{(n^{-1/d}z(\bm{H}),2 \gamma r(\bm{H})-\log n)},\label{defzeta}
\end{align}
where $\omega^{(d+1)}$ is the set of all $(d+1)$-tuples of hyperplanes of $\omega$ with pairwise distinct entries. Hence, $\zeta_{n}\equiv \zeta_n[\eta]$ is a scaling of the process of inball centers of cells, marked by a transformation of their radius. 

Let $\phi$ be the Borel measure on $\R$ given by $\varphi((c,\infty))=e^{-c}$ for all $c\in \mathbb{R}$.
We compare the process $\zeta_n$ with a Poisson process $\nu$ in the lcscH space $\mathbb{Y}:=\R^d \times \mathbb{R}$ restricted to $W \times (c,\infty)$ for some compact $W \subset \R^d$. We consider the Kantorovich-Rubinstein (KR) distance between the distributions of two finite point processes $\zeta$ and $\nu$ that is given by
\begin{align*}
	\mathbf{d_{KR}}(\zeta, \nu) := \sup_{h \in \text{Lip}(\mathbb{Y})}|\E h(\zeta)-\E h(\nu)|,
\end{align*}
where $\text{Lip}(\mathbb{Y})$ is the class of all measurable 1-Lipschitz functions $h:\widehat{\mathbf{N}}_{\mathbb{Y}} \to \R$ with respect to the total variation between measures $\omega_1, \omega_2$ on $\mathbb{Y}$ given by
\begin{align*}
	d_{\text{TV}}(\omega_1,\omega_2):=\sup |\omega_1(A)-\omega_2(A)|,
\end{align*}
where the supremum is taken over all Borel sets $A \subset \mathbb{Y}$ with $\omega_1(A), \omega_2(A)<\infty$.
The KR distance between two point processes $\xi,\zeta$ dominates their total variation distance
\begin{align*}
	\mathbf{d_{TV}}(\xi,\zeta):=\sup_{A \in \mathcal{N}_{\mathbb{Y}}} |\P(\xi \in A)-\P(\zeta \in A)|,
\end{align*}
where $\mathcal{N}_{\mathbb{Y}}$ is the standard $\sigma$-algebra on the space of $\sigma$-finite counting measures on $\mathbb{Y}$ (see \cite{decreusefond2016functional}).
\begin{theorem}\label{thmr}
	Let $\zeta_{n}\equiv \zeta_n[\eta]$ be as above and let $\nu$ be a Poisson process on $\R^d \times \mathbb{R}$ with intensity measure $\gamma^{(d)} \lambda_d \otimes \varphi$.  For all $c \in \mathbb{R}$ and all compact $W \subset \R^d$ there exists a constant $C>0$ such that
	\begin{align*}
		\mathbf{d_{KR}}(\zeta_n \cap (W \times (c,\infty)), \nu \cap (W \times (c,\infty))) \le C n^{-\delta^*} (\log n)^{d+1}
	\end{align*}
	for $n$ large enough, where $\delta^*=\delta^*(d)\in (0,1/d)$ is the solution of the fixed point equation
	\begin{align}
		\delta=\frac{\omega_{d-1}}{\omega_d} \int_{\frac{2+\delta}{3+\delta}}^1(1-x^2)^{\frac{d-3}{2}} \,\mathrm{d}x.\label{eqn:fix}
	\end{align}
\end{theorem}

In the second part of this article we study point processes of centers of cells  that are large with respect to more general functions. Let $\Sigma:\mathcal{K}^d \to \R$ be continuous (with respect to the Hausdorff metric), homogeneous of degree $k$ for some $k>0$, not identically zero and increasing under set inclusion. Such functions are called {\em size functions} in  \cite{hug2007typical}. Additionally, we assume that there is a constant $c_0>0$ such that $V_d(K) \le c_0 \Sigma(K)^{d/k}$ for all $K \in \mathcal{K}^d$, where $V_d(K)$ is the volume of $K$. Note that $\Sigma$ and $\Phi$ satisfy the sharp isoperimetric inequality
\begin{align}
	\Phi(K)\ge \tau \Sigma(K)^{1/k},\quad K \in \mathcal{K}^d, \label{iso}
\end{align}
for some constant $\tau >0$ (see \cite[Section 3]{hug2007asymptotic}). That this inequality is sharp means that there exists some $K \in \mathcal{K}^d$ such that equality holds in \eqref{iso}. Every such body is called an {\em extremal body} (for given $\Sigma$ and $\Phi$). For example, for $\Sigma=V_d$ we have $\tau=2\kappa_d^{1/d}$. We assume that all extremal bodies of $\Sigma$ are Euclidean balls. For instance, all intrinsic volumes $V_k$ ($k \in \{2,\dots,d\}$) have this property ($V_d$ is the volume and $2V_{d-1}$ is the surface area; see \cite[Section 14.2]{schneider2008stochastic}).

Let $Z$ be the typical cell of the random mosaic generated by $\eta$. It can be understood as a cell picked uniformly at random from all cells centered in a large compact region and will be defined rigorously in Section 6. Let $F$ denote the distribution function of $\Sigma(Z)$ and set $G:=1/(1-F)$. Since $F$ is continuous and strictly increasing on $[\gamma^{-1/k},\infty)$  (see Lemma \ref{sigmatrans}), $G$ is well-defined and invertible on $[\gamma^{-1/k},\infty)$. By definition of $G$, we have 
\begin{align}
	\P(G(\Sigma(Z))>u)=u^{-1},\quad u>G^{-1}(\gamma^{-1/k}), \label{parZ}
\end{align}
i.e.\ $G(\Sigma(Z))$ is Pareto(1)-distributed.

For $n \ge 1$ and $\omega \in \mathbf{N}_{\H}$ we consider the point process 
\begin{align}
	\xi_{n}[\omega]:=\frac{1}{(d+1)!} \sum_{\bm{H}\in \eta^{(d+1)}}& \mathbf{1}\{(\eta-\delta_{\bm{H}}) \cap \H_{B(\bm{H})} = \emptyset\}\, \delta_{(n^{-1/d}z(\bm{H}),n^{-1} G(\Sigma(C(\bm{H},\eta))))}.\label{defxi}
\end{align}

Let $\xi_n\equiv \xi_n[\eta]$ and $\psi$ be the Borel measure on $(0,\infty)$ given by $\psi((a,b))= a^{-1}-b^{-1}$ for $0<a<b<\infty$.

\begin{theorem} \label{thmsigma}
	Let $\xi_{n}\equiv \xi_n[\eta]$ be as above and let $\nu$ be a Poisson process on $\R^d \times (0,\infty)$ with intensity measure $\gamma^{(d)} \lambda_d \otimes \psi$. There is some $b\in (0,\delta^*)$ such that for all $c>0$ and all compact $W \subset \R^d$ there exists a constant $C>0$ such that for all $n$ large enough we have
	\begin{align*}
		\mathbf{d_{KR}}(\xi_{n} \cap (W \times (c,\infty)),\nu \cap (W \times (c,\infty)))\le Cn^{-b}.
	\end{align*}
	Here, $\delta^*=\delta^*(d)$ is the solution of the fixed point equation \eqref{eqn:fix}.
\end{theorem}
In the proof of Theorem \ref{thmsigma} we combine probabilistic bounds used in the proof of Theorem \ref{thmr} with geometric estimates. Intuitively spoken, we exploit that the shape of a large typical cell in $\textsf{m}$ is with high probability close to the shape of a ball. This fact is known as Kendall's conjecture that was answered in a series of articles (see e.g. \cite{hug2007asymptotic}, \cite{hug2007typical}). This brings us in a position where we can argue similarly to the proof of Theorem \ref{thmr}. The exponent $-b$ on the right-hand side of the bound in Theorem \ref{thmsigma} depends on the considered size function $\Sigma$ and the stability function $s$ (see Section 6). Our proof method does not allow us to make $b$ more explicit. We leave this as an open problem.

{\em Outlook.} We believe that many of the techniques developed in this article might be of an independent interest and be useful to establish further asymptotic results for (Poisson) hyperplane mosaics. One direction could be to consider score sums defined by Poisson hyperplanes and to investigate whether a central limit theorem holds for score functions depending on the cell sizes. This would generalize existing normal approximation results in \cite{heinrich2006central}. Another line of further research could be to establish Poisson process approximation for Poisson hyerplane mosaics defined in the hyperbolic space (see \cite{herold2021central} and \cite{otto2022hyperbolic} for recent results on Poisson and normal approximation in this space). 

\section{Some Integral Geometry}

\subsection{Geometry of random hyperplanes}
In the proofs of Theorem \ref{thmr} and Theorem \ref{thmsigma} we will need to control the number of hyperplanes passing through two fixed balls. The following lemma gives the probability that a random hyperplane with uniformly distributed normal vector and in distance $r>0$ from some point $z \in \R^d$ passes through a fixed ball.

\begin{lemma} \label{lem1}
	Let $w,z \in \R^d$ and $r,s>0$. Then we have
	\begin{align*}
		\int_{S^{d-1}} \mathbf{1}\{H(u,\langle z,u \rangle+ r)\in \mathbb{H}_{B(w,s)}\} \,\sigma_{d-1}(\mathrm{d}u) =\frac{\omega_{d-1}}{\omega_{d}} \int_{\max(\frac{r-s}{\|w-z\|},-1)}^{\min(\frac{r+s}{\|w-z\|},1)} (1-x^2)^{\frac{d-3}{2}} \mathrm{d}x.
	\end{align*}
\end{lemma}

\begin{proof}
	We use that the left-hand side only depends on $w$ and $z$ through $w-z$ and write out the definition of $H(u,r)$. This gives
	\begin{align*}
		&\int_{S^{d-1}} \mathbf{1}\{H(u,\langle z,u \rangle+ r)\in \mathbb{H}_{B(w,s)}\} \,\sigma_{d-1}(\mathrm{d}u)\\
		&\quad = \int_{S^{d-1}} \mathbf{1}\{H(u, r)\in \mathbb{H}_{B(w-z,s)}\} \,\sigma_{d-1}(\mathrm{d}u)\\
		&\quad=\int_{S^{d-1}} \mathbf{1}\{\langle w-z,u \rangle \in [r-s,r+s]\} \,\sigma_{d-1}(\mathrm{d}u).
	\end{align*}
	Choosing $S:=S^{d-1}\cap (\text{lin}\{w-z\})^\perp$ (where ``lin'' stands for the linear hull) in \cite[Lemma 6.5.1]{schneider2008stochastic} and substituting $x=\cos (t)$ in a second step, we find that the above is given by
	\begin{align*}
		&\frac{\omega_{d-1}}{\omega_d}\int_{S} \int_0^\pi \sin^{d-2}(t) \mathbf{1}\{\cos(t)\|w-z\| \in [r-s,r+s] \}\,\mathrm{d}t\,\sigma_{d-2}(\mathrm{d}v)\\
		&\quad =\frac{\omega_{d-1}}{\omega_d} \int_{-1}^1 (1-x^2)^{\frac{d-3}{2}} \mathbf{1}\{x\|w-z\|\in [r-s,r+s]\}\,\mathrm{d}x,
	\end{align*}
	which implies the assertion.
\end{proof}

Fixing $r>0$, $\|w-z\|=r+s$ and letting $s \to \infty$ in Lemma \ref{lem1} shows that 
\begin{align}
	1=\frac{\omega_{d-1}}{\omega_{d}} \int_{-1}^{1} (1-x^2)^{\frac{d-3}{2}} \mathrm{d}x.\label{lem1lim}
\end{align}

In the next remark we give a bound on the expected number of hyperplanes passing through two disjoint balls. Note that in dimension $d=2$, \cite[Lemma 4.2]{chenavier2016extremes} provides an upper bound on this number and that \cite{santalo2004integral} gives the exact number using an explicit geometric construction that does not seem to work in higher dimension.

\begin{remark} \rm \label{lem2}
	Let $w,z \in \R^d$ and $r,s>0$ with $r+s \le \|w-z\|$. Then we find from the definition of the measure $\mu_{d-1}$ (see \eqref{repmu}) and Lemma \ref{lem1} that
	\begin{align}
		\mu_{d-1}(\mathbb{H}_{B(z,r)}\cap \mathbb{H}_{B(w,s)})&=\frac{2 \omega_{d-1}}{\omega_d} \int_0^s  \int_{\frac{t-r}{\|w-z\|}}^{\frac{t+r}{\|w-z\|}} (1-x^2)^{\frac{d-3}{2}} \mathrm{d}x \,\mathrm{d}t\nonumber\\
		&\le \frac{2s \omega_{d-1}}{\omega_d} \int_{\frac{-r}{s+r}}^{1} (1-x^2)^{\frac{d-3}{2}} \mathrm{d}x.\label{estmuHH}
	\end{align}
	For $a\in (0,1)$ let $L(a):=\frac{\omega_{d-1}}{\omega_d} \int_{\frac{-1}{1+a}}^{1} (1-x^2)^{\frac{d-3}{2}} \mathrm{d}x$ and note that $L(a)<1$ by \eqref{lem1lim}. From \eqref{estmuHH} we find for $ar \le s$ that
	\begin{align}
		\mu_{d-1}(\mathbb{H}_{B(z,r)}\cap \mathbb{H}_{B(w,s)}) \le 2L(a)s. \label{lem2C}
	\end{align}
\end{remark}

\begin{remark} \rm
	Note that for $d=3$ we instantly find from the equality in Remark \ref{lem2} that
	\begin{align*}
		\mu_{2}(\mathbb{H}_{B(z,r)}\cap \mathbb{H}_{B(w,s)})= \frac{2rs}{\|w-z\|},\quad r,s>0,\,r+s \le \|w-z\|.
	\end{align*}
\end{remark}

\subsection{A spherical Blaschke-Petkantschin formula}
The following lemma of spherical Blaschke-Petkantschin type is a generalization of  \cite[Theorem 7.3.2]{schneider2008stochastic} to the situation where $\ell \le d$ hyperplanes are fixed and the integration over the remaining $d+1-\ell$ hyperplanes is carried out. To formulate the statement, we need to introduce some notation. For $v_1,\dots,v_\ell \in \mathbb{R}^d$ let $\nabla_\ell(v_1,\dots,v_\ell)$ denote the $\ell$-dimensional volume of the paralleliped spanned by $v_1,\dots,v_\ell$. Moreover, we define for $v_1,\dots,v_{\ell+1}\in \mathbb{R}^{d+1}$ the number 
\begin{align*}	
	\Delta_\ell(v_1,\dots,v_{\ell+1}):=\frac1 {\ell!} \nabla_\ell(v_1-v_{\ell+1},\dots,v_\ell-v_{\ell+1}).
\end{align*}
Hence, $\Delta_\ell(v_1,\dots,v_\ell)$ is the $\ell$-dimensional volume of the convex hull of $\{v_1,\dots,v_{\ell+1}\}$. In the following we use the abbreviations $u_{1:\ell}:=(u_1,\dots,u_\ell)$ and
\begin{align*}
	H(u_{\ell+1:d+1},\bm{\tau}):=(H(u_{\ell+1},\langle z,u_{\ell+1} \rangle +r),\dots,H(u_{d+1},\langle z,u_{d+1} \rangle +r)),
\end{align*}
where $z$, $\ell$ and $r$ are always clear from the context. Recall the definition of $\mathsf{P}$ from Section 2. Given $H \in \mathbb{H}$ with (unit) normal vector $u$, we write $\sigma_0$ for the (discrete) uniform distribution on $S_{H^\perp}=\{-u,u\}$. The underlying hyperplane $H$ will always be clear from the context.

\begin{lemma} \label{bplem}
	Let $1\le \ell \le d$ and $f:\H^{d+1} \to [0,\infty)$ be a measurable function. We have
	\begin{align*}
		\int_{\H^{d+1}} f \mathrm{d}\mu_{d-1}^{d+1}&= 2^{d+1} d! \int_{\H^{\ell}} \int_{\bigtimes_{i=1}^\ell S_{H_i^\perp}} \int_0^\infty \int_{(S^{d-1})^{d+1-\ell}} \int_{\bigcap_{i=1}^\ell (H_i-ru_i)}   f(H_{1:\ell},H(u_{\ell+1:d+1},\bm{\tau}))\nonumber\\
		& \times  \frac{\Delta_d(u_{1:d+1})}{\nabla_\ell(u_{1:\ell})} \mathbf{1}_{\mathsf{P}}(u_{1:d+1})\, \lambda_{d-\ell}(\mathrm{d}z) \sigma_{d-1}^{d+1-\ell}(\mathrm{d}u_{\ell+1:d+1})\, \mathrm{d}r \, \sigma_{0}^\ell(\mathrm{d}u_{1:\ell})  \, \mu_{d-1}^{\ell}(\mathrm{d}H_{1:\ell}).
	\end{align*}
\end{lemma}

\begin{proof}
	We use \cite[Theorem 7.3.2]{schneider2008stochastic} (note the missing factor $2^{d+1}$ on the right-hand side of the statement there) and obtain
	\begin{align*}
		&\int_{\H^{d+1}} f \,\mathrm{d}\mu_{d-1}^{d+1}=  2^{d+1}d! \int_{\R^d} \int_0^\infty \int_{\mathsf{P}} f(H(\bm{u},\bm{\tau})) \Delta_d(\bm{u}) \sigma_{d-1}^{d+1}(\mathrm{d}\bm{u})\, \mathrm{d}r\,\mathrm{d}z. \end{align*}
	We now replace the integration over the inner $\ell$ unit vectors $u_1,\dots,u_\ell$ by an integration over their orthogonal complements and obtain for the above
	\begin{align}
		& 2^{d+1}d! \int_{G(d,d-1)^{\ell}} \int_{\bigtimes\limits_{i=1}^\ell S_{G_i^\perp}} \int_{\R^d} \int_0^\infty \int_{S^{d+1-\ell}} f(z+G_{1:\ell}+ru_{1:\ell},H(u_{\ell+1:d+1},\bm{\tau}))\nonumber\\
		&\qquad \times \, \Delta_d(u_{1:d+1}) \mathbf{1}_\mathsf{P}(u_{1:d+1}) \sigma_{d-1}^{d+1-\ell}(\mathrm{d}u_{\ell+1:d+1})\, \mathrm{d}r\,\mathrm{d}z\, \sigma_{0}^\ell(\mathrm{d}u_{1:\ell})  \nu_{d-1}^{\ell}(\mathrm{d}G_{1:\ell}), \label{bpsph}
	\end{align}
	where $z+G_{1:\ell}+ru_{1:\ell}:=(z+G_1+ru_1,\dots,z+G_{\ell}+ru_{\ell})$ and $\nu_q$ ($q \in \{0,\dots,d\}$) is the unique Haar measure on the Grassmannian $G(d,q)$ of $q$-dimensional linear subspaces, normalized by $\nu_q(G(d,q))=1$ (see \cite[Theorem 13.2.11]{schneider2008stochastic}).
	For $k \in \N$ and $q \in [k]$ we write $b_{kq}:=\frac{\omega_{k-q+1}\cdots \omega_k}{\omega_1\cdots \omega_q}$ and $\overline{b}:=b_{d\ell}\left(\frac{b_{\ell1}}{b_{d1}}\right)^{\ell}$. Moreover, for $L \in G(d,q)$ we denote by $G(L,d-1)$ the space of all $(d-1)$-dimensional linear subspaces containing $L$ with invariant measure $\nu_{d-1}^L$ (see \cite[Section 13.2]{schneider2008stochastic}). Similarly, let $A(d,q)$ be the affine Grassmannian of $q$-dimensional affine subspaces of $\R^d$ equipped with unique Haar measure $\mu_q$, normalized by $\mu_q(\{E \in A(d,q):\,E \cap B^d \neq \emptyset\})=\kappa_{d-q}$ (see \cite[Theorem 13.2.12]{schneider2008stochastic}). Note that $A(d,d-1)=\mathbb{H}$. For an affine subspace $E \in A(d,q)$ we write $A(E,d-1)$ for the space of affine subspaces containing $E$ with invariant measure $\mu_{d-1}^{E}$. Using Theorem \cite[Theorem 7.2.5]{schneider2008stochastic} we find that \eqref{bpsph} is given by
	\begin{align*}
		&2^{d+1}d! \overline{b}\int_{G(d,d-\ell)} \int_{G(L,d-1)^{\ell}} \int_{\bigtimes\limits_{i=1}^\ell S_{G_i^\perp}} \int_{\R^d} \int_0^\infty \int_{S^{d+1-\ell}} f(z+G_{1:\ell}+ru_{1:\ell},H(u_{\ell+1:d+1},\bm{\tau})) \nonumber\\
		&\quad \times \nabla_\ell(u_{1:\ell})^{d-\ell} \Delta_d(u_{1:d+1}) \mathbf{1}_\mathsf{P}(u_{1:d+1}) \sigma_{d-1}^{d+1-\ell}(\mathrm{d}u_{\ell+1:d+1}) \mathrm{d}r\mathrm{d}z\,\sigma_{0}^\ell(\mathrm{d}u_{1:\ell})\nonumber\\
		&\quad \times (\nu_{d-1}^L)^{\ell}(\mathrm{d}G_{1:\ell}) \nu_{d-\ell}(\mathrm{d}L). 
	\end{align*}
	Now we use that $\R^d=L \oplus L^\perp$, write $\tau_i:=\langle z+z',u_i\rangle +r$ with $z \in L$ and $z' \in L^\perp$ for $i=\ell+1,\dots,d+1$ and find that the above is given by
	\begin{align*}
		&2^{d+1}d! \overline{b}\int_{G(d,d-\ell)} \int_{G(L,d-1)^{\ell}} \int_{\bigtimes\limits\limits_{i=1}^\ell S_{G_i^\perp}} \int_{L^\perp} \int_{L} \int_0^\infty \int_{S^{d+1-\ell}} f(z+G_{1:\ell}+ru_{1:\ell},H(u_{\ell+1:d+1},\bm{\tau})) \nonumber\\
		&\quad \times \nabla_\ell(u_{1:\ell})^{d-\ell} \Delta_d(u_{1:d+1}) \mathbf{1}_\mathsf{P}(u_{1:d+1}) \sigma_{d-1}^{d+1-\ell}(\mathrm{d}u_{\ell+1:d+1})\, \mathrm{d}r \,\lambda_{d-\ell}(\mathrm{d}z)\,\lambda_{\ell}(\mathrm{d}z') \sigma_{0}^\ell(\mathrm{d}u_{1:\ell})\nonumber\\ &\quad \times(\nu_{d-1}^L)^{\ell}(\mathrm{d} G_{1:\ell})		 \nu_{d-\ell}(\mathrm{d}L).
	\end{align*}
	From \cite[(13.9)]{schneider2008stochastic} we obtain for the above
	\begin{align*}
		& 2^{d+1}d! \overline{b}\int_{A(d,d-\ell)} \int_{A(E,d-1)^{\ell}} \int_{\bigtimes\limits_{i=1}^\ell S_{G_i^\perp}} \int_{E} \int_0^\infty \int_{S^{d+1-\ell}} f(H_{1:\ell}+ru_{1:\ell}) \nabla_\ell(u_{1:\ell})^{d-\ell} \Delta_d(u_{1:d+1})\nonumber\\
		&\quad \times \mathbf{1}_\mathsf{P}(u_{1:d+1}) \sigma_{d-1}^{d+1-\ell}(\mathrm{d}u_{\ell+1:d+1}) \mathrm{d}r \lambda_{d-\ell}(\mathrm{d}z)\sigma_{0}^\ell(\mathrm{d}u_{1:\ell}) (\mu_{d-1}^E)^{\ell}(\mathrm{d}H_{1:\ell}) \mu_{d-\ell}(\mathrm{d}E),
	\end{align*}
	where $H_{1:\ell}+ru_{1:\ell}:=(H_1+ru_1,\dots,H_{\ell}+ru_{\ell},H(u_{\ell+1:d+1},\bm{\tau}))$. Using \cite[Theorem 7.2.8]{schneider2008stochastic} and \cite[Lemma 14.1.1]{schneider2008stochastic}, the last term is given by
	\begin{align*}
		& 2^{d+1}d!\int_{\mathbb{H}^{\ell}} \int_{\bigtimes\limits_{i=1}^\ell S_{G_i^\perp}} \int_0^\infty \int_{S^{d+1-\ell}} \int_{\bigcap_{i=1}^\ell (H_i-ru_i)}   f(H_1,\dots,H_{\ell},H(u_{\ell+1:d+1},\bm{\tau}))\nonumber\\
		&\times \frac{\Delta_d(u_{1:d+1})}{\nabla_\ell(u_{1:\ell})} \mathbf{1}_\mathsf{P}(u_{1:d+1})\, \lambda_{d-\ell}(\mathrm{d}z) \sigma_{d-1}^{d+1-\ell}(\mathrm{d}u_{\ell+1:d+1})\, \mathrm{d}r \,\sigma_{0}^\ell(\mathrm{d}u_{1:\ell})  \, \mu_{d-1}^{\ell}(\mathrm{d}(H_1,\dots,H_{\ell})).
	\end{align*}
\end{proof}

\section{Poisson process approximation}
In this section we give an extension of \cite[Theorem 6.1]{BSY21} and use its notation. We begin with a brief repetition of the setup. Let $(\mathbb{X},\mathcal{X})$ and $(\mathbb{Y},\mathcal{Y})$ be locally compact second countable Hausdorff (lcscH) spaces. Let $g:\mathbb{X}^k \times \mathbf{N}_{\mathbb{X}} \to \{0,1\}$, $f:\mathbb{X}^k \times \mathbf{N}_{\mathbb{X}} \to \mathbb{Y}$ be measurable functions that are symmetric in the $\mathbf{x}$ coordinates and let $\mathcal{F}$ be the space of closed subsets of $\mathbb{X}$ equipped with the Fell topology. We assume that $\mathcal{S}:\mathbb{X}^k \times \mathbf{N}_{\mathbb{X}} \to \F$ is measurable and that $f,g$ are localized to $\mathcal{S}$, i.e., for all $\omega \in \mathbf{N}_{\mathbb{X}}$ and for all $S \supset \mathcal{S}(\mathbf{x},\omega)$ we have that
\begin{align}
	&g(\mathbf{x},\omega)=g(\mathbf{x},\omega \cap S),\nonumber\\
	&f(\mathbf{x},\omega)=f(\mathbf{x},\omega \cap S)\quad \text{if}\quad g(\mathbf{x},\omega)=1. \label{fgass}
\end{align}
Moreover, we assume for all $\mathbf{x}\in \mathbb{X}^k$ that $\mathcal{S}(\mathbf{x},\cdot):\mathbf{N}_{\mathbb{X}}\to \mathcal{F}$ is a stopping set. A mapping $\mathcal{S}:\mathbf{N}_{\mathbb{H}} \to \mathcal{F}$ is called {\em stopping set} if 
\begin{align}
	\{\omega \in \mathbf{N}_{\mathbb{H}}:\mathcal{S}(\omega)\subset S\}=\{\omega \in \mathbf{N}_{\mathbb{H}}:\mathcal{S}(\omega \cap S)\subset S\}\label{eqn:defstop}
\end{align}
for all compact $S \subset \mathbb{H}$.

Define 
\begin{align*}
	\xi[\omega]:=\frac{1}{k!} \sum_{\mathbf{x} \in \omega^{(k)}} g(\mathbf{x},\omega) \delta_{f(\mathbf{x},\omega)},\quad \omega \in \mathbf{N}_{\mathbb{X}},
\end{align*}
and let $\xi \equiv \xi [\eta]$, where $\eta$ is a Poisson process on $\mathbb{X}$ with $\sigma$-finite intensity measure $\mathbf{K}$. Then we obtain from the multivariate Mecke equation (see \cite[Theorem 4.4]{last2017lectures}) that the intensity measure $\mathbf{L}$ of $\xi$ is given by
\begin{align*}
	\mathbf{L}(\cdot)=\frac{1}{k!} \int_{\mathbb{X}^k} \mathbb{E} \mathbf{1}\{f(\mathbf{x},\eta+\delta_{\mathbf{x}})\in \cdot\} g(\mathbf{x},\eta+\delta_{\mathbf{x}})\,\mathbf{K}^k(\mathrm{d}\mathbf{x}).
\end{align*}

\begin{theorem} \label{PPA}
	Let $\xi$ be the process defined above with $f,g$ satisfying \eqref{fgass} and $\mathbf{L}(\mathbb{Y})< \infty$. Let $\zeta$ be a Poisson process with finite intensity measure $\mathbf{M}$. Further, suppose that we are given a measurable mapping $\mathbf{x} \mapsto S_{\mathbf{x}}$ from $\mathbb{X}^k$ to $\mathcal{F}$ satisfying $\mathbf{x} \subset S_{\mathbf{x}}$. For $\omega \in \mathbf{N}_{\mathbb{X}}$ let 
	\begin{align*}
		\tilde{g}(\mathbf{x},\omega):=g(\mathbf{x},\omega) \mathbf{1}\{\mathcal{S}(\mathbf{x},\omega) \subset S_{\mathbf{x}}\}.
	\end{align*}
	Then 
	\begin{align*}
		d_{\mathbf{KR}}(\xi,\zeta)\le d_{\text{TV}}(\mathbf{L},\mathbf{M})+E_1+E_2+E_3+E_4+E_5+E_6
	\end{align*}
	with 
	\begin{align*}
		&E_1= \frac{2}{k!} \int_{\mathbb{X}^k} \E g(\mathbf{x},\eta+\delta_{\mathbf{x}}) \mathbf{1}\{\mathcal{S}(\mathbf{x},\eta+\delta_{\mathbf{x}}) \not \subset S_{\mathbf{x}}\} \mathbf{K}^k (\mathrm{d}\mathbf{x}),\\
		&E_2=\frac{2}{(k!)^2} \int_{\mathbb{X}^k} \int_{\mathbb{X}^k} \E \mathbf{1}\{(\eta+\delta_{\mathbf{x}}+\delta_{\mathbf{z}}) \cap S_{\mathbf{x}} \cap S_{\mathbf{z}} \neq \emptyset\} \tilde{g}(\mathbf{x},\eta+\delta_{\mathbf{x}}) \E \tilde{g}(\mathbf{z},\eta+\delta_{\mathbf{z}})\\ 
		&\quad \quad \quad \quad \quad \times \,\mathbf{K}^k(\mathrm{d}\mathbf{z}) \mathbf{K}^k(\mathrm{d}\mathbf{x}),\\
		&E_3=\frac{2}{(k!)^2} \int_{\mathbb{X}^k} \int_{\mathbb{X}^k} \E \mathbf{1}\{(\eta+\delta_{\mathbf{x}}+\delta_{\mathbf{z}}) \cap S_{\mathbf{x}} \cap S_{\mathbf{z}} = \emptyset\} \tilde{g}(\mathbf{x},\eta+\delta_{\mathbf{x}})\\
		& \quad \quad \quad \times \,\E \mathbf{1}\{\eta \cap S_{\mathbf{x}} \cap S_{\mathbf{z}} \neq \emptyset\} \tilde{g}(\mathbf{z},\eta+\delta_{\mathbf{z}})  \mathbf{K}^k(\mathrm{d}\mathbf{z}) \mathbf{K}^k(\mathrm{d}\mathbf{x}),\\
		&E_4=\frac{2}{(k!)^2} \int_{\mathbb{X}^k} \int_{\mathbb{X}^k} \E \mathbf{1}\{(\eta+\delta_{\mathbf{x}}+\delta_{\mathbf{z}}) \cap S_{\mathbf{x}} \cap S_{\mathbf{z}} \neq \emptyset\} \tilde{g}(\mathbf{x},\eta+\delta_{\mathbf{x}}+\delta_{\mathbf{z}}) \tilde{g}(\mathbf{z},\eta+\delta_{\mathbf{x}}+\delta_{\mathbf{z}})\nonumber\\
		& \quad \quad \quad \times\mathbf{K}^k(\mathrm{d}\mathbf{z}) \mathbf{K}^k(\mathrm{d}\mathbf{x}),\\
		&E_5=\frac{2}{(k!)^2} \int_{\mathbb{X}^k} \int_{\mathbb{X}^k}  \E \mathbf{1}\{(\eta+\delta_{\mathbf{x}}+\delta_{\mathbf{z}}) \cap S_{\mathbf{x}} \cap S_{\mathbf{z}} = \emptyset\} \tilde{g}(\mathbf{x},\eta+\delta_{\mathbf{x}}+\delta_{\mathbf{z}}) \tilde{g}(\mathbf{z},\eta+\delta_{\mathbf{x}}+\delta_{\mathbf{z}})\\
		& \quad \quad \quad \times \,\E \mathbf{1}\{\eta\cap S_{\mathbf{x}} \cap S_{\mathbf{z}} \neq \emptyset\} \mathbf{K}^k(\mathrm{d}\mathbf{z}) \mathbf{K}^k(\mathrm{d}\mathbf{x}),\\
		&E_6=\frac{2}{k!} \sum_{\emptyset \subsetneq I \subsetneq \{1,\dots,k\}} \frac{1}{(k-|I|)!} \int_{\mathbb{X}^k} \int_{\mathbb{X}^{k-|I|}} \E  \tilde{g}(\mathbf{x},\eta+\delta_{\mathbf{x}}+\delta_{\mathbf{z}}) \tilde{g}(\mathbf{z},\eta+\delta_{\mathbf{x}}+\delta_{\mathbf{z}})\\
		& \quad \quad \quad \times \, \mathbf{K}^{k-|I|}(\mathrm{d}\mathbf{z}) \mathbf{K}^k(\mathrm{d}\mathbf{x}),
	\end{align*}
	where for $I=\{i_1,\dots,i_m\}$ we set $\mathbf{x}_I=(x_{i_1},\dots,x_{i_m})$ and $(\mathbf{x}_I,\mathbf{z})=(x_{i_1},\dots,x_{i_m},z_{1:k-m})$. 
\end{theorem}

\begin{remark} \rm \label{rem:Poi}
	(a) By interchanging the roles of $\mathbf{x}$ and $\mathbf{z}$ in the term $E_3$ in Theorem 8, one sees that $E_3 \le E_2$.\\
	(b) Note that
	\begin{align*}
		E_2+E_3&\le \frac{2}{(k!)^2} \int_{\mathbb{X}^k} \int_{\mathbb{X}^k} \mathbf{1}\{ S_{\mathbf{x}} \cap S_{\mathbf{z}} \neq \emptyset\}\E  \tilde{g}(\mathbf{x},\eta+\delta_{\mathbf{x}}) \E \tilde{g}(\mathbf{z},\eta+\delta_{\mathbf{z}}) \mathbf{K}^k(\mathrm{d}\mathbf{z}) \mathbf{K}^k(\mathrm{d}\mathbf{x})=:E_2',\\
		E_4+E_5&\le \frac{2}{(k!)^2} \int_{\mathbb{X}^k} \int_{\mathbb{X}^k} \mathbf{1}\{ S_{\mathbf{x}} \cap S_{\mathbf{z}} \neq \emptyset\}\E  \tilde{g}(\mathbf{x},\eta+\delta_{\mathbf{x}}+\delta_{\mathbf{z}}) \tilde{g}(\mathbf{z},\eta+\delta_{\mathbf{x}}+\delta_{\mathbf{z}}) \mathbf{K}^k(\mathrm{d}\mathbf{z}) \mathbf{K}^k(\mathrm{d}\mathbf{x})\\
		&=:E_3'.
	\end{align*}	
	This shows that Theorem \ref{PPA} is a refinement of \cite[Theorem 4.1]{BSY21}, where $d_{\mathbf{KR}}(\xi,\zeta)$ is bounded by $d_{TV}(\mathbf{L},\mathbf{M})+E_1+E_2'+E_3'+E_6$.
\end{remark}

\begin{proof}[Proof of Theorem \ref{PPA}]
	We proceed along the same lines as in the proof of \cite[Theorem 4.1]{BSY21}. First assume that for all $\mathbf{x} \in \mathbb{X}^k$ and $\omega \in \mathbf{N}_{\mathbb{X}}$ we have that $\mathcal{S}(\mathbf{x},\omega) \subset S_{\mathbf{x}}$. Then we have that $\tilde{g}=g$, $E_1=0$ and
	\begin{align*}
		g(\mathbf{x},\omega)=g(\mathbf{x},\omega \cap S_{\mathbf{x}}),\quad f(\mathbf{x},\omega)=f(\mathbf{x},\omega \cap S_{\mathbf{x}})\quad \text{if}\quad g(\mathbf{x},\omega)=1.
	\end{align*}
	Let $\eta'$ be a Poisson process on $\mathbb{X}$ that is independent of $\eta$ and has intensity measure $\mathbf{K}$. For fixed $h \in \text{Lip}(\mathbb{Y})$ we need to bound the difference $\E h(\zeta)-\E h(\xi)$ which is by \cite[(4.5)]{BSY21} given by
	\begin{align}
		\E h(\zeta)-\E h(\xi)=\frac{1}{k!} &\int_0^\infty \int_{\mathbb{X}^k} \left( \E g(\mathbf{x},\eta'+\delta_{\mathbf{x}}) D_{f(\mathbf{x},\eta'+\delta_{\mathbf{x}})} [P_sh(\xi[\eta])] \right.\nonumber\\
		&-\left. \E g(\mathbf{x},\eta+\delta_{\mathbf{x}}) D_{f(\mathbf{x},\eta+\delta_{\mathbf{x}})} [P_sh(\xi[\eta+\delta_{\mathbf{x}}]-\delta_{f(\mathbf{x},\eta+\delta_{\mathbf{x}})})] \right) \mathbf{K}^k(\mathrm{d}\mathbf{x})\mathrm{d}s,\label{4.5an}
	\end{align}
	where $D_xh(\omega):=h(\omega+\delta_x)-h(\omega)$ and $P_s$ is the Markov semigroup corresponding to the generator $\mathcal{L}$ that is given by
	\begin{align*}
		\mathcal{L}h(\omega):=\int_{\mathbb{X}} D_xh(\omega)\mathbf{M}(\mathrm{d}x)-\int_{\mathbb{X}} D_xh(\omega-\delta_x) \omega(\mathrm{d}x).
	\end{align*}
	
	For $\mathbf{x} \in \mathbb{X}^k$ and $\omega \in \mathbf{N}_{\mathbb{X}}$ we define
	\begin{align*}
		\xi_{\mathbf{x}}[\omega]:&=\frac{1}{k!} \sum_{\mathbf{z}\in \omega^{(k)} } \mathbf{1}\{(\omega+\delta_{\mathbf{x}}) \cap S_{\mathbf{z}} \cap S_{\mathbf{x}}=\emptyset\} g(\mathbf{z},\omega) \delta_{f(\mathbf{z},\omega)}\\
		&=\frac{1}{k!} \sum_{\mathbf{z}\in \omega^{(k)} } \mathbf{1}\{(\omega+\delta_{\mathbf{x}}) \cap S_{\mathbf{z}} \cap S_{\mathbf{x}}=\emptyset\} g(\mathbf{z},\omega \cap S_{\mathbf{z}}) \delta_{f(\mathbf{z},\omega \cap S_{\mathbf{z}})}\\
		&=\frac{1}{k!} \sum_{\mathbf{z}\in \omega^{(k)} } \mathbf{1}\{(\omega+\delta_{\mathbf{x}}) \cap S_{\mathbf{z}} \cap S_{\mathbf{x}}=\emptyset\} g(\mathbf{z},\omega \cap S_{\mathbf{z}} \cap S_{\mathbf{x}}+\omega \cap S_{\mathbf{z}}\cap S_{\mathbf{x}}^c)\\ 
		&\qquad \qquad \qquad \times \delta_{f(\mathbf{z},\omega \cap S_{\mathbf{z}}\cap S_{\mathbf{x}}+\omega \cap S_{\mathbf{z}}\cap S_{\mathbf{x}}^c)}\\
		&=\frac{1}{k!} \sum_{\mathbf{z}\in \omega^{(k)} } \mathbf{1}\{(\omega+\delta_{\mathbf{x}}) \cap S_{\mathbf{z}} \cap S_{\mathbf{x}}=\emptyset\} g(\mathbf{z},\omega \cap S_{\mathbf{x}}^c) \delta_{f(\mathbf{z},\omega \cap S_{\mathbf{x}}^c)}.
	\end{align*}
	It follows from \cite[(2.9)]{BSY21} and the multivariate Mecke equation (see \cite[Theorem 4.4]{last2017lectures} that
	\begin{align}
		&|\E g(\mathbf{x},\eta'+\delta_{\mathbf{x}})D_{f(\mathbf{x},\eta'+\delta_{\mathbf{x}})} [P_sh(\xi[\eta])] - \E g(\mathbf{x},\eta'+\delta_{\mathbf{x}})D_{f(\mathbf{x},\eta'+\delta_{\mathbf{x}})} [P_sh(\xi_{\mathbf{x}}[\eta])] |\nonumber\\
		&\quad \le \frac{2}{k!} e^{-s} \E g(\mathbf{x},\eta'+ \delta_{\mathbf{x}}) \sum_{\mathbf{z} \in \eta^{(k)}} \mathbf{1}\{(\eta+\delta_{\mathbf{x}}) \cap S_{\mathbf{z}} \cap S_{\mathbf{x}} \neq \emptyset\} g(\mathbf{z},\eta)\nonumber\\
		&\quad \le \frac{2}{k!} e^{-s}  \int_{\mathbb{X}^k} \E g(\mathbf{x},\eta+\delta_{\mathbf{x}})\E \mathbf{1}\{(\eta+\delta_{\mathbf{x}}+\delta_{\mathbf{z}}) \cap S_{\mathbf{z}} \cap S_{\mathbf{x}} \neq \emptyset\}   g(\mathbf{z},\eta+\delta_{\mathbf{z}}) \mathbf{K}^k(\mathrm{d} \mathbf{z}).\label{Ps1}
	\end{align}
	By assumption \eqref{fgass}, $f(\mathbf{x},\eta+\delta_{\mathbf{x}})$ and $g(\mathbf{x},\eta+\delta_{\mathbf{x}})$ depend only on $\eta \cap S_{\mathbf{x}}$. Since $\eta \stackrel{d}{=} \eta'$, we obtain from the independence property of the Poisson process that
	\begin{align}
		\E g(\mathbf{x},\eta'+\delta_{\mathbf{x}}) D_{f(\mathbf{x},\eta'+\delta_{\mathbf{x}})} [P_sh(\xi_{\mathbf{x}}[\eta])] = \E g(\mathbf{x},\eta+\delta_{\mathbf{x}})D_{f(\mathbf{x},\eta+\delta_{\mathbf{x}})} [P_sh(\xi_{\mathbf{x}}[\eta \cap S_{\mathbf{x}}^c+\eta'\cap S_{\mathbf{x}}])]. \label{e=e}
	\end{align}
	Note that
	\begin{align*}
		\xi_{\mathbf{x}}[\eta \cap S_{\mathbf{x}}^c+\eta'\cap S_{\mathbf{x}}]=\frac{1}{k!} \sum_{\mathbf{z} \in (\eta \cap S_{\mathbf{x}}^c)^{(k)}} \mathbf{1}\{(\eta'+\delta_{\mathbf{x}}) \cap S_{\mathbf{z}}\cap S_{\mathbf{x}}=\emptyset\} g(\mathbf{z},\eta \cap S_{\mathbf{x}}^c) \delta_{f(\mathbf{z},\eta \cap S_{\mathbf{x}}^c)}.
	\end{align*}
	Hence, we obtain that
	\begin{align}
		&\mathbf{d_{TV}}\Big(\frac{1}{k!} \sum_{\mathbf{z} \in \eta^{(k)}} g(\mathbf{z},\eta+\delta_{\mathbf{x}}) \delta_{f(\mathbf{z},\eta+\delta_{\mathbf{x}})},\xi_{\mathbf{x}}[\eta \cap S_{\mathbf{x}}^c+\eta'\cap S_{\mathbf{x}}] \Big)\nonumber\\
		&\quad \le \frac{1}{k!} \sum_{\mathbf{z} \in (\eta \cap S_{\mathbf{x}}^c)^{(k)}} \mathbf{1}\{(\eta'+\delta_{\mathbf{x}}) \cap S_{\mathbf{x}} \cap S_{\mathbf{z}} = \emptyset,\,\eta \cap S_{\mathbf{x}} \cap S_{\mathbf{z}} \neq \emptyset \} g(\mathbf{z},\eta \cap S_{\mathbf{x}}^c)\nonumber\\
		&\quad \quad +\frac{1}{k!} \sum_{\mathbf{z} \in \eta^{(k)}} \mathbf{1}\{(\eta+\delta_{\mathbf{x}}) \cap S_{\mathbf{x}} \cap S_{\mathbf{z}} = \emptyset,\,\eta' \cap S_{\mathbf{x}} \cap S_{\mathbf{z}} \neq \emptyset \} g(\mathbf{z},\eta + \delta_{\mathbf{x}})\nonumber\\
		&\quad \quad +\frac{1}{k!} \sum_{\mathbf{z} \in \eta^{(k)}} \mathbf{1}\{(\eta+\delta_{\mathbf{x}}) \cap S_{\mathbf{x}} \cap S_{\mathbf{z}} \neq \emptyset\} g(\mathbf{z},\eta+\delta_{\mathbf{x}}). \label{dtv}
	\end{align}
	As in \cite{BSY21}, we define the point process
	\begin{align*}
		\hat{\xi}_{\mathbf{x}} :=\xi[\eta+\delta_{\mathbf{x}}]-g(\mathbf{x},\eta+\delta_{\mathbf{x}})\delta_{f(\mathbf{x},\eta+\delta_{\mathbf{x}})} -\frac{1}{k!} \sum_{\mathbf{z}\in \eta^{(k)}} g(\mathbf{z},\eta+\delta_{\mathbf{x}}) \delta_{f(\mathbf{z},\eta+\delta_{\mathbf{x}})}.
	\end{align*}
	From \cite[(2.9)]{BSY21}, \eqref{dtv} and the Mecke equation we obtain that
	\begin{align*}
		&|\E g(\mathbf{x},\eta+\delta_{\mathbf{x}})D_{f(\mathbf{x},\eta+\delta_{\mathbf{x}})}[P_sh(\xi[\eta+\delta_{\mathbf{x}}]-\delta_{f(\mathbf{x},\eta+\delta_{\mathbf{x}})})]\nonumber\\
		&\quad \quad -\E g(\mathbf{x},\eta+\delta_{\mathbf{x}})D_{f(\mathbf{x},\eta+\delta_{\mathbf{x}})} [P_sh(\xi_{\mathbf{x}}[\eta \cap S_{\mathbf{x}}^c+\eta'\cap S_{\mathbf{x}}])]|\nonumber\\
		&\le 2e^{-s} \E g(\mathbf{x},\eta+\delta_{\mathbf{x}})\Big( \frac{1}{k!} \sum_{\mathbf{z}\in (\eta \cap S_{\mathbf{x}}^c)^{(k)}} \mathbf{1}\{(\eta'+\delta_{\mathbf{x}}) \cap S_{\mathbf{x}} \cap S_{\mathbf{z}}=\emptyset,\,\eta \cap S_{\mathbf{x}} \cap S_{\mathbf{z}} \neq \emptyset \} g(\mathbf{z},\eta \cap S_{\mathbf{x}}^c)\nonumber\\
		&\quad \quad \quad \quad \quad \quad \quad \quad \quad+\frac{1}{k!} \sum_{\mathbf{z}\in \eta^{(k)}} \mathbf{1}\{(\eta+\delta_{\mathbf{x}}) \cap S_{\mathbf{x}} \cap S_{\mathbf{z}}= \emptyset,\,\eta' \cap S_{\mathbf{x}} \cap S_{\mathbf{z}} \neq \emptyset \} g(\mathbf{z},\eta+\delta_x)\nonumber\\
		&\quad \quad \quad \quad \quad \quad \quad \quad \quad+  \frac{1}{k!} \sum_{\mathbf{z}\in \eta^{(k)}} \mathbf{1}\{(\eta+\delta_{\mathbf{x}}) \cap S_{\mathbf{x}} \cap S_{\mathbf{z}}\neq \emptyset\} g(\mathbf{z},\eta+\delta_x) + \hat{\xi}_{\mathbf{x}}(\mathbb{Y}) \Big). 
	\end{align*}		
	Using that $\eta \cap S_{\mathbf{x}}$ and $\eta \cap S_{\mathbf{x}}^c$ are independent point processes, we obtain from the Mecke equation that the above is bounded by
	\begin{align}
		2 e^{-s} &\left(\frac{1}{k!} \int_{\mathbb{X}^k} \E \mathbf{1}\{\eta \cap S_{\mathbf{x}} \cap S_{\mathbf{z}} \neq \emptyset,\,(\eta'+\delta_{\mathbf{x}}+\delta_{\mathbf{z}}) \cap S_{\mathbf{x}} \cap S_{\mathbf{z}} = \emptyset \} g(\mathbf{x},\eta+\delta_{\mathbf{x}}) \right.\nonumber\\
		&\quad \quad \quad \times \,g(\mathbf{z},\eta'+\delta_{\mathbf{z}})\mathbf{K}^k(\mathrm{d}\mathbf{z})\nonumber\\
		&+ \frac{1}{k!} \int_{\mathbb{X}^k} \E  \mathbf{1}\{\eta' \cap S_{\mathbf{x}} \cap S_{\mathbf{z}} \neq \emptyset,\,(\eta+\delta_{\mathbf{x}}+\delta_{\mathbf{z}}) \cap S_{\mathbf{x}} \cap S_{\mathbf{z}} = \emptyset\} g(\mathbf{x},\eta+\delta_{\mathbf{x}}+\delta_{\mathbf{z}}) \nonumber\\
		&\quad \quad \quad \times \,g(\mathbf{z},\eta+\delta_{\mathbf{x}}+\delta_{\mathbf{z}})\mathbf{K}^k(\mathrm{d}\mathbf{z})\nonumber\\
		&+ \frac{1}{k!} \int_{\mathbb{X}^k} \E  \mathbf{1}\{(\eta +\delta_{\mathbf{x}}+\delta_{\mathbf{z}}) \cap S_{\mathbf{x}} \cap S_{\mathbf{z}} \neq \emptyset \} g(\mathbf{x},\eta+\delta_{\mathbf{x}}+\delta_{\mathbf{z}}) g(\mathbf{z},\eta+\delta_{\mathbf{x}}+\delta_{\mathbf{z}})\,\mathbf{K}^k(\mathrm{d}\mathbf{z})\nonumber\\
		&+ \left.\E g(\mathbf{x},\eta+\delta_{\mathbf{x}}) \hat{\xi}_{\mathbf{x}} (\mathbb{Y})\right). \label{Ps2}
	\end{align} 
	Now we substitute \eqref{Ps1} and \eqref{Ps2}  into \eqref{4.5an} and obtain with the triangle inequality
	\begin{align*}
		&|\E h(\zeta)- \E h(\xi) |\\
		&\, \le \frac{2}{(k!)^2} \int_{\mathbb{X}^k} \int_{\mathbb{X}^k} \E \mathbf{1}\{(\eta+\delta_{\mathbf{x}}+\delta_{\mathbf{z}}) \cap S_{\mathbf{x}} \cap S_{\mathbf{z}} \neq \emptyset\} \tilde{g}(\mathbf{x},\eta+\delta_{\mathbf{x}}) \E \tilde{g}(\mathbf{z},\eta+\delta_{\mathbf{z}}) \mathbf{K}^k(\mathrm{d}\mathbf{z}) \mathbf{K}^k(\mathrm{d}\mathbf{x})\\
		&\, + \frac{2}{(k!)^2} \int_{\mathbb{X}^k} \int_{\mathbb{X}^k} \E \mathbf{1}\{(\eta+\delta_{\mathbf{x}}+\delta_{\mathbf{z}}) \cap S_{\mathbf{x}} \cap S_{\mathbf{z}} = \emptyset\} \tilde{g}(\mathbf{x},\eta+\delta_{\mathbf{x}})\\
		&\quad \quad \quad \quad \quad \quad \times \,\E \mathbf{1}\{\eta \cap S_{\mathbf{x}} \cap S_{\mathbf{z}} \neq \emptyset\} \tilde{g}(\mathbf{z},\eta+\delta_{\mathbf{z}})  \mathbf{K}^k(\mathrm{d}\mathbf{z}) \mathbf{K}^k(\mathrm{d}\mathbf{x})\\
		&\, + \frac{2}{(k!)^2} \int_{\mathbb{X}^k} \int_{\mathbb{X}^k} \E \mathbf{1}\{(\eta+\delta_{\mathbf{x}}+\delta_{\mathbf{z}}) \cap S_{\mathbf{x}} \cap S_{\mathbf{z}} \neq \emptyset\} \tilde{g}(\mathbf{x},\eta+\delta_{\mathbf{x}}+\delta_{\mathbf{z}}) \tilde{g}(\mathbf{z},\eta+\delta_{\mathbf{x}}+\delta_{\mathbf{z}})\\ 
		&\quad \quad \quad \quad \quad \quad \times \,\mathbf{K}^k(\mathrm{d}\mathbf{z}) \mathbf{K}^k(\mathrm{d}\mathbf{x})\\
		&\, + \frac{2}{(k!)^2} \int_{\mathbb{X}^k} \int_{\mathbb{X}^k}  \E \mathbf{1}\{(\eta+\delta_{\mathbf{x}}+\delta_{\mathbf{z}}) \cap S_{\mathbf{x}} \cap S_{\mathbf{z}} = \emptyset\} \tilde{g}(\mathbf{x},\eta+\delta_{\mathbf{x}}+\delta_{\mathbf{z}}) \tilde{g}(\mathbf{z},\eta+\delta_{\mathbf{x}}+\delta_{\mathbf{z}})\\
		&\quad \quad \quad \quad \quad \quad \times \,\E \mathbf{1}\{\eta\cap S_{\mathbf{x}} \cap S_{\mathbf{z}} \neq \emptyset\} \mathbf{K}^k(\mathrm{d}\mathbf{z}) \mathbf{K}^k(\mathrm{d}\mathbf{x})\\
		&\, +\frac{2}{k!} \sum_{\emptyset \subsetneq I \subsetneq \{1,\dots,k\}} \frac{1}{(k-|I|)!} \int_{\mathbb{X}^k} \int_{\mathbb{X}^{k-|I|}} \E  \tilde{g}(\mathbf{x},\eta+\delta_{\mathbf{x}}+\delta_{\mathbf{z}}) \tilde{g}(\mathbf{z},\eta+\delta_{\mathbf{x}}+\delta_{\mathbf{z}})\\ 
		&\quad \quad \quad \quad \quad \quad \quad \quad \quad \quad \quad \quad\quad \quad \quad \quad \times \,\mathbf{K}^{k-|I|}(\mathrm{d}\mathbf{z}) \mathbf{K}^k(\mathrm{d}\mathbf{x}),
	\end{align*}
	where the terms on the right-hand side are $E_2$, $E_3, E_4, E_5$ and $E_6$. The rest of the proof goes along the lines of Step 2 in the proof of \cite[Theorem 4.1]{BSY21}.
\end{proof}

\section{Proof of Theorem \ref{thmr}}
To prepare the proof of Theorem \ref{thmr}, we determine  the intensity measure $\E \zeta_n$ of $\zeta_n$. Let $A \in \mathcal{B}^d$ and $y>0$.  From  \cite[Theorem 4.1.1]{schneider2008stochastic} we find that
\begin{align}
	\E \zeta_{n}(A \times (y,\infty))&=n \gamma^{(d)} \lambda_d(A ) \mathbb{Q}_0(\{K \in \mathcal{K}^d:\,2 \gamma r(K)>y+\log n\}), \label{palmmeaspre}
\end{align}
where $\mathbb{Q}_0$ is the distribution of the typical cell of a Poisson hyperplane distribution generated from a stationary, isotropic Poisson hyperplane process of intensity $\gamma$. From \cite[Theorem 10.4.6]{schneider2008stochastic} and using \eqref{defgammad} we find that for all $R>0$,
\begin{align}
	\mathbb{Q}_0(\{K \in \mathcal{K}^d:\, r(K)>R\})=\mathrm e^{-2\gamma R}. \label{est_Q0}
\end{align}
This gives for $n$ large enough
\begin{align}
	\E \zeta_{n}(A \times (y,\infty))&= \gamma^{(d)} \lambda_d(A ) \mathrm e^{-y}.\label{intmeas}
\end{align}

The following lemma gives upper bounds for the expected number of pairs of cells with centers in some compact set $W \subset \mathbb{R}^d$ in the following two scenarios. In the first one the inradius of both cells is larger than some $R>0$. In the second scenario the distance of their centers is larger than some $D>0$ and both inradii are in the interval $(aR,R]$ for some $a\in (0,1)$. 

\begin{lemma} \label{Le:le1}
	Let $\gamma>0$, $W \subset \mathbb{R}^d$ be compact and $I \subsetneq
	\{1,\dots,d+1\}$.
	\begin{enumerate}
		\item For all $R>0$ we have	
		\begin{align*}
			&\int_{\H^{d+1}} \int_{\H^{d+1-|I|}} \mathbf{1}\{z(\bm{H}) \in W\}\, \mathbf{1}\{z(\bm{H}_I,\bm{G}) \in W\}\,\mathbf{1}\{\max\{r(\bm{H}),r(\bm{H}_I,\bm{G})\}>R \} \nonumber\\
			&\qquad  \times\,\mathrm e^{-\gamma \mu_{d-1}(\H_{B(\bm{H})}\cup \H_{B(\bm{H}_I,\bm{G})})} \,\mu_{d-1}^{d+1-|I|}(\mathrm{d}\bm{G})\,\mu_{d-1}^{d+1}(\mathrm{d}\bm{H})\nonumber\\
			&\quad \le \begin{cases}
				c_0 \lambda_d(W)^2R \mathrm{e}^{-2\gamma R},&\quad I=\emptyset,\\
				c_1  \lambda_d(W) (\mathrm{diam}(W)+R)^{d-1} \mathrm{e}^{-2\gamma R},& \quad 1\le |I| \le d.
			\end{cases}
		\end{align*}
		\item Let $D>0$, $R>0$, $a\in (0,1)$ and let $L(a)=\frac{\omega_{d-1}}{\omega_d} \int_{\frac{-1}{1+a}}^{1} (1-x^2)^{\frac{d-3}{2}} \mathrm{d}x$ as in Remark \ref{lem2}. We have
		\begin{align}
			&\int_{\H^{d+1}} \int_{\H^{d+1-|I|}} \mathbf{1}\{z(\bm{H}) \in W\}\, \mathbf{1}\{r(\bm{H})+r(\bm{H}_I,\bm{G})\le\|z(\bm{H})-z(\bm{H}_I,\bm{G})\| \le D\}\  \nonumber\\
			&\qquad \times \mathrm{e}^{-\gamma \mu_{d-1}(\H_{B(\bm{H})}\cup \H_{B(\bm{H}_I,\bm{G})})} \mathbf{1}\{r(\bm{H}) \in (aR,R], r(\bm{H}_I,\bm{G}) \in (aR,R] \} \nonumber\\ 
			&\qquad \times \mu_{d-1}^{d+1-|I|}(\mathrm{d}\bm{G})\,\mu_{d-1}^{d+1}(\mathrm{d}\bm{H})\nonumber\\
			&\quad \le  \begin{cases}
				c_2 \frac{D^{d}\lambda_d(W)}{1-L(a)}\mathrm{e}^{-2\gamma aR(2-L(a))},&\quad I=\emptyset,\\
				c_3 \frac{(D+R)^{d-1}\lambda_d(W)}{1-L(a)}\mathrm{e}^{-2\gamma aR(2-L(a))},& \quad 1\le |I| \le d.\nonumber
			\end{cases}
		\end{align}
	\end{enumerate}
	Here,  $\bm{H}_I=(H_{i_1},\dots,H_{i_m})$ for $I=\{i_1,\dots,i_m\}$ and the constants $c_0,c_1,c_2,c_3>0$ depend on the dimension $d$ and on $\gamma$.
\end{lemma}

\begin{proof}
	(a) By symmetry we can assume (at the cost of a factor 2) that $r(\bm{H})\ge r(\bm{H}_I,\bm{G})$. Using the estimate $ \mu_{d-1}(\H_{B(\bm{H})}\cup \H_{B(\bm{H}_I,\bm{G})})\ge \mu_{d-1}(\H_{B(\bm{H})})=2r(\bm{H})$ gives that the left-hand side of the statement is bounded by
	\begin{align}
		&2\int_{\H^{d+1}} \int_{\H^{d+1-|I|}} \mathbf{1}\{z(\bm{H}) \in W\}\, \mathbf{1}\{z(\bm{H}_I,\bm{G}) \in W\}\,\mathbf{1}\{r(\bm{H})>R \} \mathbf 1 \{r(\bm{H}_I,\bm{G})\le r(\bm{H})\} \nonumber\\
		&\quad  \times\,\mathrm{e}^{-2\gamma r(\bm{H})} \,\mu_{d-1}^{d+1-|I|}(\mathrm{d}\bm{G})\,\mu_{d-1}^{d+1}(\mathrm{d}\bm{H}) \label{eqn:le1agen}
	\end{align}
	Now we consider the three cases $I=\emptyset$, $|I|=1$ and $2 \le |I| \le d$ separately. If $|I|=0$ we obtain from \cite[Theorem 7.3.2]{schneider2008stochastic} and \eqref{defgammad} that \eqref{eqn:le1agen} is given by
	\begin{align*}
		\frac{4((d+1)!)^2(\gamma^{(d)})^2}{\gamma^{2d+1}} \lambda_d(W)^2 \int_{R}^\infty r  \mathrm{e}^{-2\gamma r} \mathrm{d}r=	\frac{2((d+1)!)^2(\gamma^{(d)})^2(2\gamma R+1)}{\gamma^{2d+2}} \lambda_d(W)^2 \mathrm{e}^{-2\gamma R},
	\end{align*} 
	where we recall that $\gamma^{(d)}$ is the expected number of cells centered in a Borel set with Lebesgue measure one in the mosaic generated by a Poisson hyperplane process with intensity $\gamma$. 
	
	Now let $|I|=1$ and assume that $I=\{1\}$. Then we get from Lemma \ref{bplem} (note that $\nabla_1(u)=1$ for all $u \in S^{d-1}$) that \eqref{eqn:le1agen} is given by
	\begin{align}
		2^{d+2}d!&\int_{\mathbb{H}^{d+1}} \int_{S_{H_1^\perp}} \int_{0}^{r(\bm{H})} \int_{(S^{d-1})^{d}} \int_{(H_1-su_1) \cap W} \mathbf{1}\{z(\bm{H}) \in W\} \mathbf{1}\left\{r(\bm{H})>R\right\} \mathrm e^{-2\gamma r(\bm{H})} \nonumber\\
		&\quad  \times  \Delta_d(u_{1:d+1}) \mathbf{1}_\mathsf{P}(u_{1:d+1}) \lambda_{d-1} (\mathrm{d}w) \,\sigma_{d-1}^{d}(\mathrm{d}u_{2:d+1}) \,\mathrm{d}s \,\sigma_0 (\mathrm{d}u_1) \mu_{d-1}^{d+1} (\mathrm{d}\bm{H}).\label{le1s1}
	\end{align}
	Note that since the measure $\sigma_{d-1}$ is isotropic and $\Delta_d$ is rotation-invariant, we have for all $u\in S^{d-1}$
	\begin{align}
		\int_{(S^{d-1})^d} \Delta_d(u,u_{2:d+1}) \mathbf 1_{\mathsf {P}} (u,u_{2:d+1})  \,\sigma_{d-1}^d(\mathrm{d}u_{2:d+1})&=	\int_{\mathsf {P}} \Delta_d(u_{1:d+1})  \,\sigma_{d-1}^{d+1}(\mathrm{d}u_{1:d+1})\nonumber\\
		&=\frac{(d+1)1^{(d)}}{2^d},\label{eqn:Deltad}
	\end{align}
	where the second equality holds by \eqref{defgammad}. Since $\lambda((H_1-su_1) \cap W) \le \mathrm{diam}(W)^{d-1}$ for all $H_1 \in \mathbb{H}$, $s>0$ and $u_1 \in S^{d-1}$, we find that \eqref{le1s1} is bounded by
	\begin{align*}
		4 (d+1)! 1^{(d)}\mathrm{diam}(W)^{d-1} \int_{\mathbb{H}^{d+1}}    \mathbf{1}\{z(\bm{H}) \in W\} \mathbf{1}\left\{r(\bm{H})>R\right\} r(\bm{H}) \mathrm e^{-2\gamma r(\bm{H})}   \mu_{d-1}^{d+1} (\mathrm{d}\bm{H}),
	\end{align*}
	and the claim now follows after an application of \cite[Theorem 7.3.2]{schneider2008stochastic}. 
	
	If $2 \le |I|\le d$ we use that the triangle inequality gives 
	\begin{align}
		d(G_i,z(\bm{H}))\le r(\bm{H}_I,\bm{G})+\|z(\bm{H}_I,\bm{G})-z(\bm{H})\|,\quad i \in \{1,\dots,d+1\} \setminus I. \label{eqn:tridis}
	\end{align}
	This yields for \eqref{eqn:le1agen} the bound
	\begin{align*}
		&2\int_{\H^{d+1}} \mu_{d-1}(\mathbb{H}_{B(o,r(\bm{H})+\text{diam}(W))})^{d+1-|I|} \mathbf{1}\{z(\bm{H}) \in W\}\,\,\mathbf{1}\{r(\bm{H})>R \} \mathrm{e}^{-2\gamma r(\bm{H})} \,\mu_{d-1}^{d+1}(\mathrm{d}\bm{H}).
	\end{align*}
	Since $\mu_{d-1}(\mathbb{H}_{B(z,r)})=2r$ for all $z\in \mathbb{R}^d$ and all $r>0$, the above is given by
	\begin{align*}
		2^{d+2-|I|} \int_{\H^{d+1}} (r(\bm{H})+\text{diam}(W))^{d+1-|I|} \mathbf{1}\{z(\bm{H}) \in W\}\,\mathbf{1}\{r(\bm{H})>R \} \mathrm{e}^{-2\gamma r(\bm{H})} \,\mu_{d-1}^{d+1}(\mathrm{d}\bm{H}),
	\end{align*}
	and the claim again follows after an application of \cite[Theorem 7.3.2]{schneider2008stochastic}.
	
	(b)
	For $r(\bm{H}) \le R$ and $ aR< r(\bm{H}_I,\bm{G})$ we have $ar(\bm{H}) \le r(\bm{H}_I,\bm{G})$.  Hence, we obtain from Remark \ref{lem2} that $\mu_{d-1}(\H_{B(\bm{H})}\cap \H_{B(\bm{H}_I,\bm{G})}) \le 2L(a)r(\bm{H}_I,\bm{G})$. Consequently, by the inclsuion-exclusion principle,
	\begin{align*}
		\mu_{d-1}(\H_{B(\bm{H})}\cup \H_{B(\bm{H}_I,\bm{G})}) \ge 2r(\bm{H})+2(1-L(a))r(\bm{H}_I,\bm{G}).
	\end{align*}
	Therefore, the left-hand side in the statement of part (b) of the lemma is bounded by \begin{align}
		\int_{\H^{d+1}}& \int_{\H^{d+1-|I|}} \mathbf{1}\{z(\bm{H}) \in W\}\, \mathbf{1}\{\|z(\bm{H})-z(\bm{H}_I,\bm{G})\| \le D\}\ \mathrm{e}^{-2\gamma r(\bm{H})-2\gamma(1-L(a))r(\bm{H}_I,\bm{G})} \nonumber\\
		&\, \times\,  \mathbf{1}\left\{r(\bm{H}) \in (aR,R], r(\bm{H}_I,\bm{G}) \in (aR,R] \right\}  \,\mu_{d-1}^{d+1-|I|}(\mathrm{d}\bm{G})\,\mu_{d-1}^{d+1}(\mathrm{d}\bm{H}).\label{eqn:Leb}
	\end{align}
	
	If $I=\emptyset$ we apply \cite[Theorem 7.3.2]{schneider2008stochastic} twice and find that \eqref{eqn:Leb} is bounded by
	\begin{align*}
		\frac{\kappa_d((d+1)!)^2(\gamma^{(d)})^2}{\gamma^{2d+2}(1-L(a))} D^d \lambda_d(W) \mathrm{e}^{2\gamma aR(2-L(a))}.
	\end{align*}
	
	For $|I|=1$ we obtain from Lemma \ref{bplem} and \eqref{eqn:Deltad} for \eqref{eqn:Leb} the bound
	\begin{align*}
		\frac{2 (d+1)! 1^{(d)}}{\gamma (1-L(a))} D^{d-1} \int_{\mathbb{H}^{d+1}}    \mathbf{1}\{z(\bm{H}) \in W\} \mathbf{1}\left\{r(\bm{H})>aR\right\} r(\bm{H}) \mathrm e^{-2\gamma r(\bm{H})}   \mu_{d-1}^{d+1} (\mathrm{d}\bm{H}).
	\end{align*}
	From here the assertion follows by \cite[Theorem 7.3.2]{schneider2008stochastic}
	
	If $2\le |I| \le d$ we find using \eqref{eqn:tridis} that \eqref{eqn:Leb} is bounded by
	\begin{align*}
		\mathrm e^{-2\gamma a R(1-L(a))}&\int_{\H^{d+1}}\mu_{d-1}(\mathbb{H}_{B(z(\bm{H}),R+D)})^{d+1-|I|} \mathbf{1}\{z(\bm{H}) \in W\}\mathbf{1}\left\{r(\bm{H})> aR \right\}  \, \mathrm{e}^{-2\gamma r(\bm{H})} \\
		&\times \mu_{d-1}^{d+1}(\mathrm{d}\bm{H}).
	\end{align*}
	The claim now follows from the fact that $\mu_{d-1}(\mathbb{H}_{B(z,r)})=2r$ for all $z\in \mathbb{R}^d$ and all $r>0$ and from \cite[Theorem 7.3.2]{schneider2008stochastic}.
\end{proof}

\begin{proof}[Proof of Theorem \ref{thmr}]
	Let $c\in \mathbb{R}$, $W \subset \R^d$ be compact and let $\nu$ be a Poisson process on $\R^d \times \mathbb{R}$ with intensity measure $\gamma^{(d)} \lambda_d \otimes \varphi$, where $\varphi$ is given by $\varphi((y,\infty))=e^{-y}$ for all $y\in\mathbb{R}$. Since by  \eqref{intmeas} the intensity measures of $\zeta_n \cap (W \times (c,\infty))$ and $\nu \cap (W \times (c,\infty))$ coincide, their total variation is zero for $n$ large enoguh. We apply Theorem \ref{PPA} with
	\begin{align*}
		&g(\bm{H},\omega)=\mathbf{1}\{n^{-1/d}z(\bm{H}) \in W\} \mathbf{1}\{(\omega-\delta_{\bm{H}}) \cap \mathbb{H}_{B(\bm{H})}=\emptyset\}\mathbf{1}\{2 \gamma r(\bm{H})-\log n >c\},\\
		&f(\bm{H},\omega)=(n^{-1/d}z(\bm{H}),2\gamma r(\bm{H})-\log n ),\quad \omega \in \mathbf{N}_{\mathbb{H}},\, \bm{H} \in \omega^{(d+1)},
	\end{align*}
	and the (deterministic) stopping set $\mathcal{S}$ given by $\mathcal{S}(\bm{H},\omega)=\mathbb{H}_{B(\bm{H})}$ for $\bm{H} \in \mathbb{H}^{d+1}$ in general position and all $\omega \in \mathbf{N}_\H$. Letting $S_{\bm{H}}=\mathbb{H}_{B(\bm{H})}$ gives $g=\tilde{g}$ and, therefore, $E_1=0$. Moreover, since  $g(\bm{G},\omega+\delta_{\bm{G}})=1$ implies that $\omega\cap \mathbb{H}_{B(\bm{G})}=\emptyset$, we find that $E_3=E_4=0$. This yields the bound
	\begin{align}
		&\mathbf{d_{KR}}(\zeta_n \cap (W \times (c,\infty)), \nu \cap (W \times (c,\infty))) \le E_2+E_5+E_6
	\end{align}
	with
	\begin{align*}
		E_2&=\frac{2\gamma^{2d+2}}{((d+1)!)^2} \int_{\H^{d+1}} \int_{\H^{d+1}} \mathbf{1}\{\delta_{(\bm{H},\bm{G})}\cap \H_{B(\bm{H})} \cap \H_{B(\bm{G})} \neq \emptyset\} \mathbf{1}\{z(\bm{H}) \in W_n,\,z(\bm{G}) \in W_n\} \\
		& \times \,\mathbf{1}\left\{r(\bm{H})>\frac{c+\log n}{2\gamma}\right\} \mathbf{1}\left\{r(\bm{G})>\frac{c+\log n}{2\gamma}\right\}\,\mathrm{e}^{-2\gamma r(\bm{H})}\,\mathrm{e}^{-2\gamma r(\bm{G})}\,\mu_{d-1}^{d+1}(\mathrm{d}\bm{G})\,\mu_{d-1}^{d+1}(\mathrm{d}\bm{H}),\\
		E_5&=\frac{2\gamma^{2d+2}}{((d+1)!)^2} \int_{\H^{d+1}} \int_{\H^{d+1}} \P(\eta \cap \H_{B(\bm{H})} \cap  \H_{B(\bm{G})} \neq \emptyset)  \mathbf{1}\{z(\bm{H}) \in W_n,\,z(\bm{G}) \in W_n\} \\
		& \times \,\mathbf{1}\left\{r(\bm{H})>\frac{c+\log n}{2\gamma}\right\} \mathbf{1}\left\{r(\bm{G})>\frac{c+\log n}{2\gamma}\right\}\,\mathrm{e}^{-\gamma \mu_{d-1}(\H_{B(\bm{H})}\cup \H_{B(\bm{G})})}\\
		& \times \mathbf{1}\{r(\bm{H})+r(\bm{G})\le\|z(\bm{H})-z(\bm{G})\|\} \,\mu_{d-1}^{d+1}(\mathrm{d}\bm{G})\,\mu_{d-1}^{d+1}(\mathrm{d}\bm{H}),\\
		E_6&=\frac{2\gamma^{d+1}}{(d+1)!} \sum_{\emptyset \subsetneq I \subsetneq \{1,\dots,d+1\}} \frac{\gamma^{|I|}}{(d+1-|I|)!}\int_{\H^{d+1}} \int_{\H^{d+1-|I|}} \mathbf{1}\{z(\bm{H}) \in W_n,\,z(\bm{H}_I,\bm{G}) \in W_n\} \\
		&  \times \,\mathbf{1}\left\{r(\bm{H})>\frac{c+\log n}{2\gamma}\right\} \mathbf{1}\left\{r(\bm{H}_I,\bm{G})>\frac{c+\log n}{2\gamma}\right\}\,\mathrm{e}^{-\gamma \mu_{d-1}(\H_{B(\bm{H})}\cup \H_{B(\bm{H}_I,\bm{G})})}\\
		& \times \mathbf{1}\{r(\bm{H})+r(\bm{H}_I,\bm{G})\le\|z(\bm{H})-z(\bm{H}_I,\bm{G})\|\}\, \mu_{d-1}^{d+1-|I|}(\mathrm{d}\bm{G})\,\mu_{d-1}^{d+1}(\mathrm{d}\bm{H}),
	\end{align*}
	where $W_n:=n^{1/d}W=\{x \in \R^d:\,x \in n^{1/d}W\}$  $\bm{H}_I=(H_{i_1},\dots,H_{i_m})$ for $I=\{i_1,\dots,i_m\}$. Moreover, we have used that $\delta_{(\bm{H}_{|I|+1:d+1},\bm{G})}\cap \H_{B(\bm{H}) \cup B(\bm{H}_I,\bm{G})} = \emptyset$ implies that $r(\bm{H})+r(\bm{H}_I,\bm{G})\le\|z(\bm{H})-z(\bm{H}_I,\bm{G})\|$ for $I \subsetneq \{1,\dots,d+1\}$. 
	
	Next we consider the terms $E_2,E_5,E_6$ separately. In the following, $c_i>0$ $(i \in \mathbb{N})$ are positive constants. Their exact values are not important for the argument.
	
	{\em The estimate of $E_2$.} Note that $$\mathbf 1\{\delta_{(\bm{H},\bm{G})}\cap \H_{B(\bm{H})} \cap \H_{B(\bm{G})} \neq \emptyset\}\le \sum_{i=1}^{d+1} \mathbf 1\{H_i \in \mathbb{H}_{B(\bm{G})}\}+  \sum_{i=1}^{d+1} \mathbf 1\{G_i \in \mathbb{H}_{B(\bm{H})}\}.$$ From symmetry in $\bm{H}$ and $\bm{G}$ we find that $E_2$ is bounded by
	\begin{align*}
		&\frac{4\gamma^{2d+2}}{(d+1)!d!} \int_{\H^{d+1}} \int_{\H^{d+1}}  \mathbf 1\{G_1 \in \H_{B(\bm{H})}\} \mathbf{1}\{z(\bm{H}) \in W_n\} \mathbf{1}\{z(\bm{G}) \in W_n\}\, \mathbf{1}\left\{r(\bm{H})>\frac{c+\log t}{2\gamma}\right\} \\
		&\quad  \times \,\mathbf{1}\left\{r(\bm{G})>\frac{c+\log n}{2\gamma}\right\} \mathrm{e}^{-2\gamma r(\bm{H})}\,\mathrm{e}^{-2\gamma r(\bm{G})}\,\mu_{d-1}^{d+1}(\mathrm{d}\bm{G})\,\mu_{d-1}^{d+1}(\mathrm{d}\bm{H}),
	\end{align*}
	where $\bm{G}=(G_1,\dots,G_{d+1})$. Next we apply Lemma \ref{bplem} with $\ell=1$ to the inner integral (note that $\nabla_1(v)=1$ for all $v \in S^{d-1}$) and obtain
	\begin{align*}
		&\frac{2^{d+3} \gamma^{2d+2} }{(d+1)!} \int_{\H^{d+1}} \int_{\H_{B(\bm{H})}} \int_{G_1^\perp } \int_{(S^{d-1})^{d}} \int_{\frac{c+\log n}{2\gamma}}^\infty  \int_{W_n\cap(G_1-r u_1) } \mathbf{1}\{z(\bm{H}) \in W_n\} \nonumber\\
		&\quad\times\, \mathbf{1}\left\{r(\bm{H})>\frac{c+\log n}{2\gamma}\right\}\, \mathrm{e}^{-2\gamma r(\bm{H})} \mathrm{e}^{-2\gamma s} \Delta_d(u_{1:d+1}) \mathbf{1}_\mathsf{P}(u_{1:d+1})\\
		&\quad \times \,\lambda_{d-1}(\mathrm{d}w)\, \mathrm{d}s \,\sigma_{d-1}^{d}(\mathrm{d}u_{2:d+1})\,\sigma_1(\mathrm{d} u_1)\,\mu_{d-1}(\mathrm{d}G_1)\,\mu_{d-1}^{d+1}(\mathrm{d}\bm{H}).
	\end{align*}
	Since $\lambda_{d-1}(W_n\cap(G_1-rv_1)) \le n^{(d-1)/d} \text{diam}(W)^{d-1} $, \eqref{defgammad} and  $\mu_{d-1}(\mathbb{H}_{B(\bm{H})})=2r(\bm{H})$ (see \eqref{repmu}), the above is bounded by
	\begin{align*}
		\frac{4\gamma^{d+1} \gamma^{(d)}\mathrm{diam}(W)^{d-1} \mathrm{e}^{-c}}{ n^{1/d}d!}& \int_{\H^{d+1}} r(\bm{H}) \mathbf{1}\{z(\bm{H}) \in W_n\}\, \mathbf{1}\left\{r(\bm{H})>\frac{c+\log n}{2\gamma}\right\} \mathrm{e}^{-2\gamma r(\bm{H})} \\
		&\quad \times \mu_{d-1}^{d+1}(\mathrm{d}\bm{H}).
	\end{align*}
	From \cite[Theorem 7.3.2]{schneider2008stochastic} we conclude that $E_2$ is bounded by $c_1  n^{-1/d} \log n.$ 
	
	{\em The estimate of $E_5$.}
	Given $\delta>0$ (to be specified later), we split $E_5$ into the sum 
	\begin{align}
		&\frac{2\gamma^{2d+2}}{((d+1)!)^2} \int_{\H^{d+1}} \int_{\H^{d+1}} \mathbf{1}\{z(\bm{H}) \in W_n\}\, \mathbf{1}\{z(\bm{G}) \in W_n\}\nonumber\\
		&\quad \times \mathbf{1}\left\{r(\bm{H})>\frac{c+\log n}{2\gamma}\right\} \mathbf{1}\left\{r(\bm{G})>\frac{c+\log n}{2\gamma}\right\}\nonumber\\
		&\quad  \times \mathbf{1}\{\max\{r(\bm{H}),r(\bm{G})\}>\frac{(2+\delta) (c+\log n)}{2\gamma} \text{ or } \|z(\bm{H})-z(\bm{G})\|\le (\log n)^{(d+1)/d}\}\,\nonumber\\
		&\quad \times\,\mathrm{e}^{-\gamma \mu_{d-1}(\H_{B(\bm{H})}\cup \H_{B(\bm{G})})}\mathbf{1}\{r(\bm{H})+r(\bm{G})\le\|z(\bm{H})-z(\bm{G})\|\} \,\mu_{d-1}^{d+1}(\mathrm{d}\bm{G})\,\mu_{d-1}^{d+1}(\mathrm{d}\bm{H})\label{rE5s1}\\
		&+\frac{2\gamma^{2d+2}}{((d+1)!)^2} \int_{\H^{d+1}} \int_{\H^{d+1}} \P(\eta \cap \H_{B(\bm{H})} \cap  \H_{B(\bm{G})} \neq \emptyset) \mathbf{1}\{z(\bm{H}) \in W_n\}\, \mathbf{1}\{z(\bm{G}) \in W_n\} \nonumber\\
		&\quad  \times \,\mathbf{1}\left\{r(\bm{H})>\frac{c+\log n}{2\gamma}\right\} \mathbf{1}\left\{r(\bm{G})>\frac{c+\log n}{2\gamma}\right\}\nonumber\\
		&\quad \times \mathbf{1}\{\max\{r(\bm{H}),r(\bm{G})\} \le \frac{(2+\delta) (c+\log n)}{2\gamma}\}\, \mathbf{1}\{\|z(\bm{H})-z(\bm{G})\|> (\log n)^{(d+1)/d}\}\nonumber\\
		&\quad \times \mathrm{e}^{-\gamma \mu_{d-1}(\H_{B(\bm{H})}\cup \H_{B(\bm{G})})} \mathbf{1}\{r(\bm{H})+r(\bm{G})\le\|z(\bm{H})-z(\bm{G})\|\} \,\mu_{d-1}^{d+1}(\mathrm{d}\bm{G})\,\mu_{d-1}^{d+1}(\mathrm{d}\bm{H}).\label{rE5s2}
	\end{align}
	Note that \eqref{rE5s1} is bounded by
	\begin{align}
		&\int_{\H^{d+1}} \int_{\H^{d+1}} \mathbf{1}\{z(\bm{H}) \in W_n\}\, \mathbf{1}\{z(\bm{G}) \in W_n\}\,\mathbf{1}\left\{\max\{r(\bm{H}),r(\bm{G})\}>\frac{(2+\delta) (c+\log n)}{2\gamma}\right\} \nonumber\\
		&\quad  \times\,\mathrm{e}^{-\gamma \mu_{d-1}(\H_{B(\bm{H})}\cup \H_{B(\bm{G})})} \,\mu_{d-1}^{d+1}(\mathrm{d}\bm{G})\,\mu_{d-1}^{d+1}(\mathrm{d}\bm{H})\label{rE5s2a}\\
		&\,+\int_{\H^{d+1}} \int_{\H^{d+1}} \mathbf{1}\{z(\bm{H}) \in W_n\}\, \mathbf{1}\{\|z(\bm{H})-z(\bm{G})\| \le (\log n)^{(d+1)/d}\}\,\mathrm{e}^{-\gamma \mu_{d-1}(\H_{B(\bm{H})}\cup \H_{B(\bm{G})})}  \nonumber\\
		&\quad  \times\,\mathbf{1}\left\{\frac{(2+\delta) (c+\log n)}{2 \gamma}\ge r(\bm{G})>\frac{c+\log n}{2 \gamma},\,\frac{(2+\delta) (c+\log n)}{2 \gamma} \ge r(\bm{H})>\frac{c+\log n}{2 \gamma}\right\} \nonumber\\
		&\quad \times \mathbf{1}\{r(\bm{H})+r(\bm{G})\le\|z(\bm{H})-z(\bm{G})\|\} \,\mu_{d-1}^{d+1}(\mathrm{d}\bm{G})\,\mu_{d-1}^{d+1}(\mathrm{d}\bm{H}). \label{rE5s2b}
	\end{align}
	From Lemma \ref{Le:le1}(a) with $I=\emptyset$ and $R:= \frac{(2+\delta) (c+\log n)}{2 \gamma}$ we get for \eqref{rE5s2a} the bound
	$c_2 n^{-\delta} \log n$. By part (b) of the same lemma with $a:=(2+\delta)^{-1}$ and $D:=(\log n)^{(d+1)/d}$ we conclude that \eqref{rE5s2b} is bounded by $c_3 n^{-(1-L(\frac{1}{2+\delta}))} (\log n)^{d+1}$. Hence, letting $\delta^*$ be the solution of the fixed point equation
	\begin{align}
		\delta=1-L\left(\frac{1}{2+\delta}\right):=\frac{\omega_{d-1}}{\omega_d} \int_{\frac{2+\delta}{3+\delta}}^1(1-x^2)^{\frac{d-3}{2}} \,\mathrm{d}x, \label{fixequ}
	\end{align}
	we find that \eqref{rE5s1} is bounded by $c_3 n^{-\delta^*} (\log n)^{d+1}$.
	
	Now we discuss \eqref{rE5s2}. If $r(\bm{G})\le r(\bm{H})$, \eqref{estmuHH} gives 
	\begin{align}
		\mu_{d-1}(\H_{B(\bm{H})}\cap \H_{B(\bm{G})}) =\frac{2 \omega_{d-1}}{\omega_d} \int_0^{r(\bm{G})}  \int_{\frac{s-r(\bm{H})}{\|z(\bm{H})-z(\bm{G})\|}}^{\frac{s+r(\bm{H})}{\|z(\bm{H})-z(\bm{G})\|}} (1-x^2)^{\frac{d-3}{2}} \mathrm{d}x \,\mathrm{d}s\le \frac{4\omega_{d-1} r(\bm{H})r(\bm{G})}{\omega_d \|z(\bm{H})-z(\bm{G})\|}, \label{mubou}
	\end{align}
	where the inequality is obtained by bounding the integrand by 1. Hence, for $\max\{r(\bm{H}),r(\bm{G})\}\le \frac{(2+\delta) (c+\log n)}{2\gamma}$ and $\|z(\bm{H})-z(\bm{G})\|> (\log n)^{(d+1)/d}$, we obtain from the inclusion-exclusion principle the bound 
	\begin{align}
		\mu_{d-1}(\H_{B(\bm{H})}\cup \H_{B(\bm{G})}) \ge 2\gamma r(\bm{H})+2\gamma r(\bm{G})-c_3. \label{mulowbou}
	\end{align}
	This helps us as follows to bound \eqref{rE5s2}. Assuming that $r(\bm{G})\le r(\bm{H})$ (at the cost of a factor 2), we use the bound $ \P(\eta \cap \H_{B(\bm{H})} \cap  \H_{B(\bm{G})} \neq \emptyset) \le \gamma \mu_{d-1}(\H_{B(\bm{H})}\cap \H_{B(\bm{G})})$ together with \eqref{mubou}. To bound the exponential function in the integrand of \eqref{rE5s2} we use \eqref{mulowbou}. This gives for \eqref{rE5s2} the bound
	\begin{align}
		&\frac{16 \omega_{d-1}\gamma^{2d+3}e^{c_3}}{\omega_d((d+1)!)^2} \int_{\H^{d+1}} \int_{\H^{d+1}} \frac{r(\bm{H})^2}{\|z(\bm{H})-z(\bm{G})\|} \mathbf{1}\{z(\bm{H}) \in W_n\}\, \mathbf{1}\{z(\bm{G}) \in W_n\} \nonumber\\
		&\quad  \times \,\mathbf{1}\left\{r(\bm{H})\ge r(\bm{G})>\frac{c+\log n}{2\gamma}\right\} \,\mathrm{e}^{-2\gamma r(\bm{H})-2 \gamma r(\bm{G})} \,\mu_{d-1}^{d+1}(\mathrm{d}\bm{G})\,\mu_{d-1}^{d+1}(\mathrm{d}\bm{H}).\label{rE5s22}
	\end{align}
	From \cite[Theorem 7.3.2]{schneider2008stochastic} and \eqref{est_Q0} we obtain that the above is given by
	\begin{align*}
		\frac{16\omega_{d-1} \gamma^{(d)}\gamma^{d+2}}{n\omega_d(d+1)!e^{-c_3}} \int_{\H^{d+1}}& \int_{W_n} \frac{ r(\bm{H})^2}{\|z(\bm{H})-z\|}  \mathbf{1}\Big\{z(\bm{H}) \in W_n\Big\} \mathbf{1}\Big\{r(\bm{H})>\frac{c+\log n}{2\gamma}\Big\} \,\mathrm{e}^{-2\gamma r(\bm{H})}\\ 
		&\quad \times\mathrm{d}z\,\mu_{d-1}^{d+1}(\mathrm{d}\bm{H}).
	\end{align*}
	We introduce spherical coordinates around $z(\bm{H})$ in the inner integration, let $w:=\mathrm{diam}(W)$ and obtain the bound
	\begin{align*}
		&\frac{16\omega_{d-1}\gamma^{(d)}\gamma^{d+2}}{n(d+1)!e^{-c_3}} \int_{\H^{d+1}} \	\int_{0}^{n^{1/d}w} t^{d-2}  r(\bm{H})^2 \mathbf{1}\Big\{z(\bm{H}) \in W_n\Big\} \mathbf{1}\Big\{r(\bm{H})>\frac{c+\log n}{2\gamma}\Big\}\\
		&\qquad \qquad\qquad\qquad\qquad\qquad\times \mathrm{e}^{-2\gamma r(\bm{H})}\mathrm{d}t\mu_{d-1}^{d+1}(\mathrm{d}\bm{H})\\
		&	\quad =\frac{16\omega_{d-1}\gamma^{(d)}\gamma^{d+2} w^{d-1}}{n^{1/d}(d-1)(d+1)!e^{-c_3}}  \int_{\H^{d+1}} r(\bm{H})^2 \mathbf{1}\Big\{z(\bm{H}) \in W_n\Big\} \mathbf{1}\Big\{r(\bm{H})>\frac{c+\log n}{2\gamma}\Big\}\\
		&\qquad \qquad\qquad\qquad\qquad\qquad\times \mathrm{e}^{-2\gamma r(\bm{H})}\mu_{d-1}^{d+1}(\mathrm{d}\bm{H}).
	\end{align*}
	From \cite[Theorem 7.3.2]{schneider2008stochastic} and \eqref{est_Q0} we conclude that \eqref{rE5s2} is bounded by $c_{4}  n^{-1/d}(\log n)^2$.

	{\em The estimate of $E_6$.}
	Analogously to $E_5$, we split $E_6$ into the sum
	\begin{align}
		&\frac{2\gamma^{d+1}}{(d+1)!}\sum_{\emptyset \subsetneq I \subsetneq \{1,\dots,d+1\}} \frac{\gamma^{|I|}}{(d+1-|I|)!}\int_{\H^{d+1}} \int_{\H^{d+1-|I|}} \mathbf{1}\{z(\bm{H}) \in W_n\}\, \mathbf{1}\{z(\bm{H}_I,\bm{G}) \in W_n\} \nonumber\\
		&\quad  \times \,\mathbf{1}\left\{r(\bm{H})>\frac{c+\log n}{2\gamma}\right\} \mathbf{1}\left\{r(\bm{H}_I,\bm{G})>\frac{c+\log n}{2\gamma}\right\}\,\mathrm{e}^{-\gamma \mu_{d-1}(\H_{B(\bm{H})}\cup \H_{B(\bm{H}_I,\bm{G})})} \nonumber\\
		&\quad \times \mathbf{1}\{\max\{r(\bm{H}),r(\bm{H}_I,\bm{G})\}>\frac{(2+\delta) (c+\log n)}{2\gamma} \text{ or } \|z(\bm{H})-z(\bm{H}_I,\bm{G})\|\le  (\log n)^2\}\nonumber\\
		&\quad \times \mathbf{1}\{r(\bm{H})+r(\bm{H}_I,\bm{G})\le\|z(\bm{H})-z(\bm{H}_I,\bm{G})\|\}\, \mu_{d-1}^{d+1-|I|}(\mathrm{d}\bm{G})\,\mu_{d-1}^{d+1}(\mathrm{d}\bm{H}) \label{rE6s1}\\
		&\,+\frac{2\gamma^{d+1}}{(d+1)!}\sum_{\emptyset \subsetneq I \subsetneq \{1,\dots,d+1\}} \frac{\gamma^{|I|}}{(d+1-|I|)!}\int_{\H^{d+1}} \int_{\H^{d+1-|I|}} \mathbf{1}\{z(\bm{H}) \in W_n\}\, \mathbf{1}\{z(\bm{H}_I,\bm{G}) \in W_n\} \nonumber\\
		&\quad  \times \,\mathbf{1}\left\{r(\bm{H})>\frac{c+\log n}{2\gamma}\right\} \mathbf{1}\left\{r(\bm{H}_I,\bm{G})>\frac{c+\log n}{2\gamma}\right\}\,\mathrm{e}^{-\gamma \mu_{d-1}(\H_{B(\bm{H})}\cup \H_{B(\bm{H}_I,\bm{G})})} \nonumber\\
		&\quad \times \mathbf{1}\{\max\{r(\bm{H}),r(\bm{H}_I,\bm{G})\}\le \frac{(2+\delta) (c+\log n)}{2\gamma},\, \|z(\bm{H})-z(\bm{H}_I,\bm{G})\|> (\log n)^2\}\nonumber\\
		&\quad \times \mathbf{1}\{r(\bm{H})+r(\bm{H}_I,\bm{G})\le\|z(\bm{H})-z(\bm{H}_I,\bm{G})\|\}\, \mu_{d-1}^{d+1-|I|}(\mathrm{d}\bm{G})\,\mu_{d-1}^{d+1}(\mathrm{d}\bm{H}). \label{rE6s2}
	\end{align}
	Argueing analogously to the bound of \eqref{rE5s1}, we obtain that \eqref{rE6s1} is bounded by $c_{5}  n^{-\delta^*} \log n$.
	
	Now we consider \eqref{rE6s2}. Using \eqref{mulowbou} we obtain that  \eqref{rE6s2} is bounded by
	\begin{align}
		&\frac{4\gamma^{d+1}e^{c_3}}{(d+1)!}\sum_{\emptyset \subsetneq I \subsetneq \{1,\dots,d+1\}} \frac{\gamma^{d+1-|I|}}{(d+1-|I|)!}\int_{\H^{d+1}} \int_{\H^{d+1-|I|}} \mathbf{1}\{z(\bm{H}) \in W_n\}\, \mathbf{1}\{z(\bm{H}_I,\bm{G}) \in W_n\} \nonumber\\
		&\quad  \times \mathbf{1}\left\{r(\bm{H})>r(\bm{H}_I,\bm{G})>\frac{c+\log n}{2\gamma}\right\}\,\mathrm{e}^{-2 \gamma r(\bm{H})-2\gamma r(\bm{H}_I,\bm{G})} \mu_{d-1}^{d+1-|I|}(\mathrm{d}\bm{G})\,\mu_{d-1}^{d+1}(\mathrm{d}\bm{H}).\label{rE6s11}
	\end{align}
	Now we distinguish the cases $|I|=1$ and $2 \le |I| \le d$. For every summand in \eqref{rE6s11} with $|I|=1$ we obtain from Lemma \ref{bplem} with $\ell=1$
	\begin{align}
		& 2^{d+1}\gamma^{d} \int_{\mathbb{H}^{d+1}} \int_{H_1^\perp } \int_{\frac{c+\log n}{2\gamma}}^{r(\bm{H})} \int_{(S^{d-1})^{d}} \int_{(H_1-su_1) \cap W_n}  \mathbf{1}\{z(\bm{H}) \in W_n\}  \mathbf{1}\left\{r(\bm{H})>\frac{c+\log n}{2\gamma}\right\} \nonumber\\
		&\qquad \times \mathrm{e}^{-2 \gamma r(\bm{H})-2\gamma s} \Delta_d(u_{1:d+1}) \mathbf{1}_\mathsf{P}(u_{1:d+1})  \lambda_{d-1} (\mathrm{d}w) \,\sigma_{d-1}^{d}(\mathrm{d}u_{2:d+1}) \,\mathrm{d}s \,\sigma_1 (\mathrm{d}u_{1}) \mu_{d-1}^{d+1} (\mathrm{d}\bm{H})\nonumber\\
		& \, \le (d+1) 1^{(d)} w^{d-1}\gamma^{d-1} n^{-1/d} \int_{\mathbb{H}^{d+1}} \mathbf{1}\{z(\bm{H}) \in W_n\}  \mathbf{1}\left\{r(\bm{H})>\frac{c+\log n}{2\gamma}\right\} r(\bm{H})\nonumber\\
		&\qquad \qquad\qquad\qquad\qquad\qquad\qquad\qquad\times \mathrm{e}^{-2 \gamma r(\bm{H})}  \mu_{d-1}^{d+1} (\mathrm{d}\bm{H}), \label{rtv-1est1}
	\end{align}
	where $w:=\textrm{diam}(W)$ and we have used the bound $\lambda_{d-1}((H_1-su_1)\cap W_n)\le \textrm{diam}(W)n^{\frac{d-1}{d}}$ and \eqref{defgammad}. An application of \cite[Theorem 7.3.2]{schneider2008stochastic} gives that \eqref{rtv-1est1} is bounded by $c_6 n^{-1/d}\log n$.
	
	For $2 \le |I| \le d$ we exploit the bound
	\begin{align*}
		d(G_i,o)\le d(G_i,z(\bm{H}_I,\bm{G}))+ d(z(\bm{H}_I,\bm{G}),o)\le r(\bm{H})+\textrm{diam}(W_n),\quad i \in \{1,\dots,d+1\}\setminus I.
	\end{align*}
	This yields for every summand in \eqref{rtv-1est1} with $2 \le |I| \le d$ the bound
	\begin{align*}
		\frac{\gamma^{d+1-|I|}}{(d+1-|I|)!}&\int_{\H^{d+1}} \mu_{d-1}(\mathbb{H}_{B(o,r(\bm{H})+\textrm{diam}(W_n))})^{d+1-|I|} \mathbf{1}\{z(\bm{H}) \in W_n\}\\
		&\times  \mathbf{1}\left\{r(\bm{H})>\frac{c+\log n}{2\gamma}\right\}\,\mathrm{e}^{-2 \gamma r(\bm{H})} \,\mu_{d-1}^{d+1}(\mathrm{d}\bm{H}).
	\end{align*}
	Using that $\mu_{d-1}(\mathbb{H}_{B(o,r)})=2r$ for $r>0$ and letting $w:=\textrm{diam}(W)$ we obtain the bound
	\begin{align*}
		\frac{(2\gamma)^{d+1-|I|}}{(d+1-|I|)!}&\int_{\H^{d+1}} \Big(r(\bm{H})+\frac w{n^{1/d}}\Big)^{d+1-|I|} \mathbf{1}\{z(\bm{H}) \in W_n\} \mathbf{1}\left\{r(\bm{H})>\frac{c+\log n}{2\gamma}\right\}\\
		&\quad \times\,\mathrm{e}^{-2 \gamma r(\bm{H})} \,\mu_{d-1}^{d+1}(\mathrm{d}\bm{H}).
	\end{align*}	
	By another application of \cite[Theorem 7.3.2]{schneider2008stochastic} we arrive at the bound $c_7 n^{-1/d} \log n$.
	
	It remains to show that the solution $\delta^*=\delta^*(d)$ of \eqref{fixequ} is less than or equal to $1/d$ for all $d\ge 2$. For $d=2$ this can be checked easily. Since $1-L(\frac{1}{2+\delta})$ is decreasing in $\delta$, it suffices for $d\ge 3$ to show that $1-L(\frac{d}{2d+1})\le \frac1d$. To establish this, we first bound $\frac{\omega_{d-1}}{\omega_d}$. Since $\omega_d:=\frac{2\pi^{d/2}}{\Gamma(d/2)}$ for $d \in \mathbb{N}$, $\Gamma(n+\frac12)=\frac{(2n)!\sqrt{\pi}}{4^nn!}$ for $n \in \mathbb{N}$ and using Stirling's formula (\cite{robbins1955remark}) 
	\begin{align*}
		\sqrt{2 \pi n}\, e^{1/(12n+1)} (n/\mathrm{e})^n \le n! \le \sqrt{2 \pi n} \,e^{1/(12n)} (n/\mathrm{e})^n,
	\end{align*}
	we find that $\frac{\omega_{d-1}}{\omega_d} \le \sqrt{\frac{d-1}{2\pi}}$. Using that $x\mapsto (1-x^2)^{\frac{d-3}{2}}$ is for $d \ge 3$ decreasing on the interval  $[0,1]$, we thus obtain
	\begin{align*}
		1-L\Big(\frac{1}{2+\delta}\Big)=\frac{\omega_{d-1}}{\omega_d} \int_{\frac{2d+1}{3d+1}}^{1} (1-x^2)^{\frac{d-3}{2}}\mathrm{d}x\le \sqrt{\frac{d-1}{2\pi}} \frac{d}{3d+1}\Big(1-\Big(\frac{2d+1}{3d+1}\Big)^2\Big)^{\frac{d-3}{2}}.
	\end{align*}
	Since $\frac{d}{3d+1}\le \frac 13$ and $1-(\frac{2d+1}{3d+1})^2\le \frac 59$ for $d \in \mathbb{N}$, the above is bounded by
	\begin{align*}
		\frac{9\sqrt{d-1}}{5\sqrt{10\pi} } \Big(\frac{\sqrt 5}{3}\Big)^d.
	\end{align*}
	Since $\frac{9d\sqrt{d-1}}{5\sqrt{10\pi} } (\frac{\sqrt 5}{3})^d <1$ for $d \in \mathbb{N}$, we deduce that $\delta^*(d)\in (0,1/d)$.
	
	We conclude that for $n$ large enough
	\begin{align*}
		\mathbf{d_{KR}}(\zeta_n \cap (W \times (c,\infty)), \nu \cap (W \times (c,\infty))) \le C  n^{-\delta^*} (\log n)^{d+1},
	\end{align*}
	where $\delta^*>0$ is the solution of the fixed point equation \eqref{fixequ} and the constant $C>0$ depends on $d$, $\gamma$, $W$ and $c$.
\end{proof}

\section{Asymptotic shape of the typical cell}
For the proof of Theorem \ref{thmsigma} we need the notion of the typical cell $Z$ of the random mosaic generated by $\eta$. That is any random polytope with distribution $\mathbb{Q}_0$ on $\mathcal{K}^d$ given by
\begin{align}
	\gamma^{(d)} \mathbb{Q}_0(\cdot)=\sum_{\bm{H}\in \eta^{(d+1)}} \mathbf 1\{(\eta-\delta_{\bm{H}}) \cap \mathbb{H}_{B(\bm{H})}=\emptyset\} \mathbf 1\{C(\bm{H},\eta)-z(\bm{H})\in \cdot\} \mathbf 1\{z(\bm{H})\in [0,1]^d\}.\label{eqn:typZ}
\end{align}
For an explicit integral representation of $\mathbb{Q}_0$ we refer to \cite[Theorem 10.4.6]{schneider2008stochastic}. The following lemma ensures that the distribution function $F$ of $\Sigma(Z)$ is strictly increasing and continuous on $[\gamma^{-1/k},\infty)$. This implies  that the function $G:=\frac{1}{1-F}$ introduced in Section 2 is well-defined and invertible on $[\gamma^{-1/k},\infty)$.

\begin{lemma} \label{sigmatrans}
	Let $Z$ be the typical cell in the random mosaic generated by a stationary Poisson hyperplane process with intensity $\gamma>0$ and let $\Sigma$ be a $k$-homogeneous size function. On the interval $[\gamma^{-1/k},\infty)$, the distribution $\mathbb{P}_{\Sigma(Z)}$ of $\Sigma(Z)$ and the Lebesgue measure $\lambda_1$ are equivalent.
\end{lemma}

\begin{proof}
	The proof follows the strategy of Section 9 in \cite{hug2007asymptotic}. We first show that $\mathbb{P}_{\Sigma(Z)}$ is absolutely continuous with respect to $\lambda_1$ on $[\gamma^{-1/k},\infty)$. For $K \in \mathcal{K}^d$ let $D(K)$ be the diameter of $K$ and $\Delta(K):=\frac{D(K)}{c_1\Sigma(K)^{1/k}}$ be the {\em relative diameter}, where $c_1$ is chosen such that $\Delta(K)\ge 1$ for all $K \in \mathcal{K}^d$. For $a>0$, $h>0$ and $m \in \mathbb{N}$ let
	\begin{align*}
		\mathcal{K}_{a,h}(m):=\{K \in \mathcal{K}^d:\,\Sigma(K)\in a(1,1+h),\,\Delta(K)\in [m,m+1)\}
	\end{align*}
	and $q_{a,h}(m):=\mathbb{P}(\Sigma(Z)\in \mathcal{K}_{a,h}(m))$. Fix $\kappa>0$. Analogously to the argumentation before $(41)$ in \cite{hug2007asymptotic} (with \cite[Lemma 5]{hug2007asymptotic} replaced by \cite[Lemma 4.5]{hug2007typical}), we find some $\nu\in \mathbb{N}$ such that
	\begin{align*}
		q_{a,1}(m)\le c(\kappa)m^{d\nu} \exp(-(1-\kappa/4)\tau d^{1/k}\gamma).
	\end{align*}
	From \cite[Lemma 4.3]{hug2007typical} we obtain $q_{a,1}(m)\le c_2\exp(-c_3ma^{1/k}\gamma)$ for some $c_2,c_3 >0$. By \cite[Lemma 4.8]{hug2007typical} we have for some $c_4>0$
	\begin{align*}
		\mathbb{P}(\Sigma(Z)\in a(1,1+h))&=\sum_{m \in \mathbb{N}} q_{a,h}(m)=c_4h a^{1/k} \gamma \Big(\sum_{m\le m_0} q_{a,1}(m)+\sum_{m>m_0} q_{a,1}(m)\Big),
	\end{align*}
	where $m_0 \in \mathbb{N}$ is chosen such that 
	\begin{align}
		c_3 m \ge 2(1-\kappa/4)\tau,\quad m >m_0. \label{eqn:estm}
	\end{align}
	Now we can argue analogously to the proof of Proposition 7.1 in \cite{HRS2004limit} (where \cite[(24)]{HRS2004limit} is replaced by \eqref{eqn:estm}) to arrive that
	\begin{align}
		\mathbb{P}(\Sigma(Z)\in a(1,1+h))\le c_5(\kappa) h \exp(-(1-\kappa/2)\tau a^{1/k}\gamma).\label{eqn:estah}
	\end{align}
	As in \cite[Section 9]{hug2007asymptotic} we can now conclude from \eqref{eqn:estah} that if a set $M \subset [\gamma^{-1/k}, \infty)$ is covered by countably many intervals of total length $\varepsilon$, then the $\mathbb{P}_{\Sigma(Z)}$-measure of $M$ is at most $c_6\varepsilon$, where $c_6$ does not depend on $\varepsilon$. This proves that $\mathbb{P}_{\Sigma(Z)}$ is absolutely continuous with respect to $\lambda_1$ on $[\gamma^{-1/k},\infty)$.
	
	Next we show that the Radon-Nikod\'{y}m density of $\mathbb{P}_{\Sigma(Z)}$ with respect to the Lebesgue measure $\lambda_1$ is positive on $[\gamma^{-1/k}, \infty)$. From \cite[Lemma 4.1]{hug2007typical} we obtain for $a \ge \gamma^{-1/k}$
	\begin{align*}
		\frac{\mathrm{d}\mathbb{P}_{\Sigma(Z)}}{\mathrm{d}\lambda_1} (a) =\lim_{h \downarrow 0} \frac{\mathbb{P}(\Sigma(Z) \in a(1,1+h))}{ah}\ge \frac ca \exp(-(1+\beta) \tau a^{1/k}\gamma)>0.
	\end{align*}
	This gives that $\mathbb{P}_{\Sigma(Z)}$ and $\lambda_1$ are equivalent measures on $[\gamma^{-1/k}, \infty)$.
\end{proof}

In order to measure the deviation of the shape of a convex body $K$ from the shape of a Euclidean ball, we use the deviation function
\begin{align}
	\vt(K):=\min\left\{\frac{R-r}{R+r}: rB^d +z \subset K \subset RB^d+z,\,r,R>0,\,z\in\R^d \right\}  \label{defvt}
\end{align}
from \cite{hug2007typical} and note that $\vt(K)=0$ if and only if $K$ is a Euclidean ball. Let $\Sigma$ be a size function as defined in Section 2. By \cite[Theorem]{hug2007typical}, there exists a continuous function $s:\mathbb{R}^+ \to \mathbb{R}^+$ with $s(\varepsilon)>0$ for all $\varepsilon>0$ and $s(0)=0$ and some constant $c_0>0$ (depending only on $\tau$) such that
\begin{align}
	\P(\vt(Z)\ge \varepsilon \mid \Sigma(Z)>u) \le c_1 \exp(-c_0s(\varepsilon)u^{1/k}\gamma),\quad u>0,\label{Ken}
\end{align}
where $c_1>0$ depends only on $\Sigma, s, \varepsilon$. In \cite{hug2007typical} the function $s$ is called a {\em stability function} for $\Sigma$ and $\vt$.

The distributions of the typical cell $Z$ and the zero cell $Z_0$ (this is the (a.s. unique) cell containing the origin $o \in \R^d$) are linked via
\begin{align}
	\P(\Sigma(Z)\ge u)\le \P(\Sigma(Z_0)\ge u),\quad u>0,\label{cellsmon}
\end{align}
which is a direct consequence of \cite[Lemma 3.1]{hug2007typical}. Nevertheless, the distributions of $\Sigma(Z)$ and $\Sigma(Z_0)$ show the same asymptotic behavior on a logarithmic scale, i.e. 
\begin{align}
	\lim_{u \to \infty} u^{-1/k} \log \P(\Sigma(Z)\ge u)=\lim_{u \to \infty} u^{-1/k} \log \P(\Sigma(Z_0)\ge u)=-\tau \gamma, \label{thmZ}
\end{align}
where $\tau$ is the constant from \eqref{iso}. This result is a direct consequence of \cite[Theorem 2]{hug2007asymptotic} and \cite[Lemma 4.1]{hug2007typical}.

Next we determine the asymptotic behavior of $G^{-1}$. Let $y>0$ and note that the fact that $F$ has unbounded support implies that $G^{-1}(ny)^{-1/k}\to \infty$ as $n\to \infty$. Hence, we obtain from \eqref{thmZ} and the definition of $G$ that 
\begin{align}
	\lim_{n \to \infty} \frac{\log n}{G^{-1}(ny)^{1/k}}=\lim_{n \to \infty}	G^{-1} (ny)^{-1/k} \log(1-F(G^{-1}(ny)))\frac{\log (ny)}{\log (1-F(G^{-1}(ny)))}=\tau \gamma. \label{asyyn}
\end{align}

\section{Stopping sets and decay of correlation}
In the proof of Theorem \ref{thmsigma} we need to control the circumradius of the typical cell $Z$. Following \cite[Section 6.3]{penrose2007gaussian}, for $\alpha \in (0,\frac{\pi}{6})$ and $z \in \R^d$ we let $K'_i(z),\,i \in \mathcal I,$ be a finite collection of infinite open cones with apex $z$, angular radius $\frac \alpha 2$ and union $\R^d$.  Let $\omega \in \mathbf{N}_\H$ be locally finite and let $R_i(z,\omega)$ be the minimal $r>0$ such that there is $u \in S^{d-1}\cap K_i(z)$ with $H(u,\langle z,u \rangle+r )\in \omega$ if such $r>0$ exists and set $R_i(z,\omega):=\infty$, otherwise. 

For $\bm{H} \in \omega^{(d+1)}$ in general position let $R(\bm{H},\omega):= (\cos \alpha)^{-1} \max_{1\le i \le I} R_i(z(\bm{H}),\omega \cap \H^{B(\bm{H})})$. If $(\omega-\delta_{\bm{H}})\cap \mathbb{H}_{B(\bm{H})}=\emptyset$, then $R(\bm{H},\omega)$ is an upper bound for the circumradius of $C(\bm{H},\omega)$ and we have
\begin{align}
	C(\bm{H},\omega)=C(\bm{H},\omega_{B(z(\bm{H}),R(\bm{H},\omega))}). \label{relCS}
\end{align}
Next we determine the distribution of $R(\bm{H},\eta)$. For all $u>(\cos \alpha)^{-1} r(\bm{H})$ we obtain
\begin{align}
	\P(R(\bm{H},\eta)>u)  &= \prod_{i \in \mathcal I}\P(R_i(z(\bm{H}),\eta \cap \H^{B(\bm{H})})>u \cos \alpha)\nonumber\\
	&= 1-(1-\mathrm{e}^{-2\gamma (u \cos \alpha -r(\bm{H}))/|\mathcal I|})^{|\mathcal I|}.\label{btR}
\end{align}

In the proof of Theorem \ref{thmsigma} we will use the stopping set $\omega \mapsto \mathbb{H}_{B(z(\bm{H}),R(\bm{H},\omega))}$ for $\bm{H} \in \mathbb{H}^{d+1}$ in general position. In the following lemma we prove that this is indeed a stopping set.

\begin{lemma}\label{lemstop}
	For all $\bm{H} \in \H^{d+1}$ in general position the mapping $\mathcal{S}(\bm{H},\cdot):\mathbf{N}_\H\to \F$ given by 
	\begin{align}
		\mathcal{S}(\bm{H},\omega):=\mathbb{H}_{B(z(\bm{H}),R(\bm{H},\omega))} \label{defS}
	\end{align}
	is a stopping set.
\end{lemma}

\begin{proof}
	We need to verify \eqref{eqn:defstop}. Let $\bm{H} \in \H^{d+1}$ be in general position, $\omega \in \mathbf{N}_{\mathbb{H}}$ and  $S \subset \mathbb{H}$ be compact. Assume that $\mathcal{S}(\bm{H},\omega) \subset S$ and let $z:=z(\bm{H})$. Then we have
	\begin{align*}
		\mathcal{S}(\bm{H},\omega)&=\mathbb{H}_{B(z(\bm{H}),(\cos \alpha)^{-1}\max_{i\in \mathcal I} \inf\{r>0:\,\exists  u \in K_i(z)\cap S^{d-1}:\,H(u,\langle z,u\rangle+r)\in \omega\})} \cap S\\
		&=\mathbb{H}_{B(z(\bm{H}),(\cos \alpha)^{-1}\max_{i\in \mathcal I} \inf\{r>0:\,\exists  u \in K_i(z)\cap S^{d-1}:\,H(u,\langle z,u\rangle+r)\in \omega \cap S\})}\\
		&=\mathcal{S}(\bm{H},\omega\cap S).
	\end{align*}
	Since $\mathcal{S}(\bm{H},\omega_1)\subset \mathcal{S}(\bm{H},\omega_2)$ for $\omega_1 \ge \omega_2$, $\mathcal{S}(\bm{H},\omega \cap S) \subset S$ implies that $\mathcal{S}(\bm{H},\omega) \subset S$. Hence, the proof is complete.
\end{proof}

In the following we consider the events that two cells (whose circumradii are not too large) are large with respect to $\Sigma$. We will prove that the correlation of these two events decays when the distance of the centers of the two cells becomes large. This will be an important step towards the proof of Theorem \ref{thmsigma}. For $z\in \mathbb{R}^d$, $\alpha \in (0,\frac{\pi}{6})$ and $i \in \mathcal I$ let $N(i):=\{j \in \mathcal I:\, K_j(z) \cap K_i(z)\neq \emptyset\}$ be the set of indices $j \in \mathcal I$ for which the cone $K_j(z)$ intersects $K_i(z)$. Let $R_i'(z,\omega):=\min \{R_j(z,\omega):\,j \in N(i)\}$ (see Figure 1). For $\bm{H} \in \mathbb{H}^{d+1}$ in general position let $R'(\bm{H},\omega):=(\cos 3\alpha)^{-1}\max_{i\in I}R_i'(z(\bm{H}),\omega)$ .
\begin{figure}[bt]
	\centering
	\includegraphics[scale=1.1]{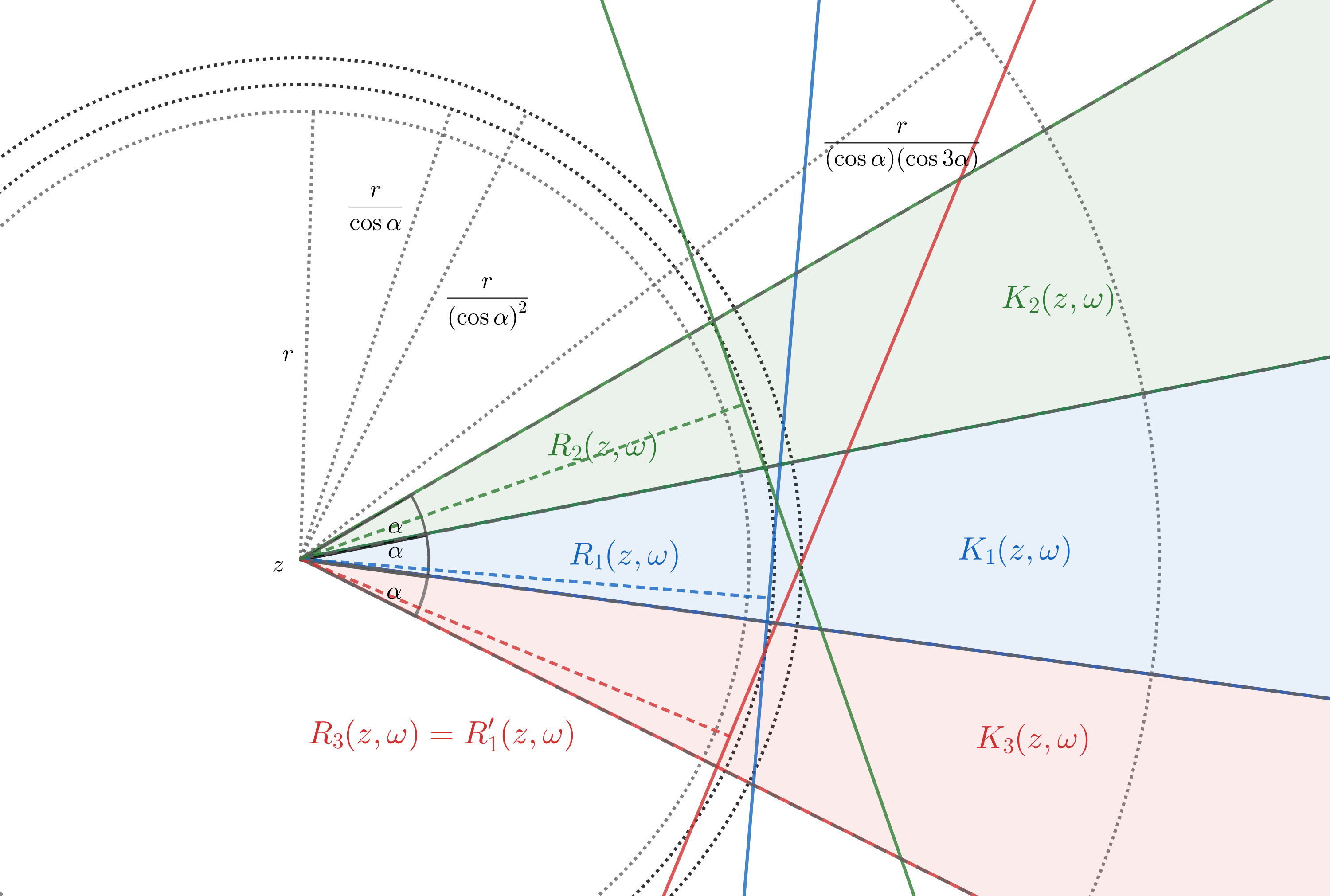}
	\caption{Here we have $N(1)=\{1,2,3\}$. Since $\min\limits_{j \in N(1)} R_j(z,\omega)=R_3(z,\omega)$, we find that $R_3(z,\omega)=R_1'(z,\omega)$.}
\end{figure}

For $r>0$ let $r_\alpha:=\frac{r}{(\cos \alpha)(\cos 3\alpha)}$. Let $I \subsetneq \{1,\dots,d+1\}$, $\bm{H}\in \mathbb{H}^{d+1}$, $\bm{G}\in \mathbb{H}^{d+1-|I|}$ such that $\bm{H}$ and $(\bm{H}_I,\bm{G})$ are in general position and define
$$\mathbb{H}(\bm{H},\bm{G},\alpha):=	\mathbb{H}_{B(z(\bm{H}),r(\bm{H})_\alpha)} \cap \mathbb{H}_{B(z(\bm{H}_I,\bm{G}),r(\bm{H}_I,\bm{G})_\alpha)}.$$Note that if $\frac{\max(r(\bm{H}),r(\bm{H}_I,\bm{G}))}{\|z(\bm{H})-z(\bm{H}_I,\bm{G})\|}$ is small enough, then for every $i \in \mathcal I$ we find some $j \in N(i)$ such that
\begin{align*}
	R_{j}(z(\bm{H}),\omega \cap \mathbb{H}(\bm{H},\bm{G},\alpha)^c)&=R_{j}(z(\bm{H}),\omega),\quad \omega \in \mathbf{N}_{\mathbb{H}}.
\end{align*}
This implies that $R_i'(z(\bm{H}),\omega  \cap \mathbb{H}(\bm{H},\bm{G},\alpha)^c)\le \max_{j \in N(i)} R_j(z(\bm{H}),\omega),\,i \in \mathcal I,$ and, hence,  
\begin{align}
	R'(\bm{H},\omega  \cap \mathbb{H}(\bm{H},\bm{G},\alpha)^c)\le \frac{\cos \alpha}{\cos 3\alpha} R(\bm{H},\omega) \label{eqn:Rbou}
\end{align}
and the statement also holds if $\bm{H}$ is replaced by $(\bm{H}_I,\bm{G})$.

Fix $\alpha \in (0,\frac \pi 6)$ and $u>0$. For $\bm{H} \in \omega^{(d+1)}$ define
\begin{align*}
	h(\bm{H},\omega)&=\mathbf 1 \{(\omega-\delta_{\bm{H}})\cap \mathbb{H}_{B(\bm{H})}=\emptyset\} \1 \{\Sigma(C(\bm{H},\omega))>u\}\, \mathbf 1 \Big\{R(\bm{H},\omega)\le  \frac{r(\bm{H})}{(\cos \alpha)^{2}}\Big\},\\
	h'(\bm{H},\omega)&=\mathbf 1 \{(\omega-\delta_{\bm{H}})\cap \mathbb{H}_{B(\bm{H})}=\emptyset\} \1 \{\Sigma(C(\bm{H},\omega))>u\}\, \mathbf 1\Big \{R'(\bm{H},\omega)\le  \frac{r(\bm{H})}{(\cos \alpha) (\cos 3\alpha)} \Big\}.
\end{align*}
Using \eqref{eqn:Rbou} together with the inequality $\Sigma(C(\bm{H},\omega))\le \Sigma(C(\bm{H},\omega \cap  \mathbb{H}(\bm{H},\bm{G},\alpha)^c))$, we obtain 
\begin{align}
	&h(\bm{H},\omega)\le h'(\bm{H},\omega \cap \mathbb{H}(\bm{H},\bm{G},\alpha)^c). \label{eq:gtildepr}
\end{align} 
In the next lemma we exploit \eqref{eq:gtildepr} to derive an approximative decorrelation inequality for $h$.
\begin{lemma} \label{Le:E2bou} \rm 
	Let $I \subsetneq \{1,\dots,d+1\}$, $\bm{H}\in \mathbb{H}^{d+1}$, $\bm{G}\in \mathbb{H}^{d+1-|I|}$ such that $\bm{H}$ and $(\bm{H}_I,\bm{G})$ are in general position. Let $H \in \mathbb{H}_{B(\bm{H})} \cap \mathbb{H}_{B(\bm{H}_I,\bm{G})}$ and $\mu \in \{0, \delta_H\}$. If $\frac{\max(r(\bm{H}),r(\bm{H}_I,\bm{G}))}{\|z(\bm{H})-z(\bm{H}_I,\bm{G})\|}$ is small enough, we have
	\begin{align*}
		\mathbb{E} h(\bm{H},\eta_{\bm{H},\bm{G}} + \mu)\, h((\bm{H}_I,\bm{G}),\eta_{\bm{H},\bm{G}}+ \mu)	\le 2 \mathbb{E} h'(\bm{H},\eta_{\bm{H}})\, \mathbb{E}h'((\bm{H}_I,\bm{G}),\eta_{\bm{H}_I,\bm{G}}).
	\end{align*}
\end{lemma}

\begin{proof} 
	We apply \eqref{eq:gtildepr} with $\omega:=\eta_{\bm{H},\bm{G}}+\mu$ and obtain
	\begin{align}
		&\mathbb{E} h(\bm{H},\eta_{\bm{H},\bm{G}} + \mu)\, h((\bm{H}_I,\bm{G}),\eta_{\bm{H},\bm{G}}+ \mu)\nonumber\\
		&\quad \le 	\mathbb{E}	h'(\bm{H},\eta_{\bm{H}}\cap \mathbb{H}(\bm{H},\bm{G},\alpha)^c)\, h'((\bm{H}_I,\bm{G}),\eta_{\bm{H}_I,\bm{G}} \cap \mathbb{H}(\bm{H},\bm{G},\alpha)^c).\label{eq:gtildepr2}
	\end{align}
	Note that $h'(\bm{H},\eta_{\bm{H}} \cap \mathbb{H}(\bm{H},\bm{G},\alpha)^c)$ is measurable with respect to $\eta \cap \mathbb{H}_{B(z(\bm{H}),r(\bm{H})_\alpha)}$ and that $h'((\bm{H}_I,\bm{G}),\eta_{\bm{H}_I,\bm{G}}  \cap \mathbb{H}(\bm{H},\bm{G},\alpha)^c)$ is measurable with respect to $\eta \cap \mathbb{H}^{B(z(\bm{H}),r(\bm{H})_\alpha)}$. Hence, by independence of the processes $\eta \cap \mathbb{H}_{B(z(\bm{H}),r(\bm{H})_\alpha)}$ and  $\eta \cap \mathbb{H}^{B(z(\bm{H}),r(\bm{H})_\alpha)}$, \eqref{eq:gtildepr2} factorizes into
	\begin{align*}
		\mathbb{E}	h'(\bm{H},\eta_{\bm{H}}\cap \mathbb{H}(\bm{H},\bm{G},\alpha)^c)\, \mathbb{E}h'(\bm{H}_I,\bm{G},\eta_{\bm{H}_I,\bm{G}} \cap \mathbb{H}(\bm{H},\bm{G},\alpha)^c).
	\end{align*} 
	Thus it remains to show that $ \mathbb{E} h'(\bm{H},\eta_{\bm{H}} \cap \mathbb{H}(\bm{H},\bm{G},\alpha)^c)\le \sqrt 2 \,\mathbb{E} h'(\bm{H},\eta_{\bm{H}} )$ to conclude the proof. Let $p:=\mathbb{P}(\eta \cap \mathbb{H}(\bm{H},\bm{G},\alpha)=\emptyset)=\exp(-\gamma \mu_{d-1}(\mathbb{H}(\bm{H},\bm{G},\alpha)))$. Since the processes $\eta \cap \mathbb{H}(\bm{H},\bm{G},\alpha)$ and  $\eta \cap \mathbb{H}(\bm{H},\bm{G},\alpha)^c$	are independent, we obtain
	\begin{align*}
		\mathbb{E} h'(\bm{H},\eta_{\bm{H}} \cap \mathbb{H}(\bm{H},\bm{G},\alpha)^c)&=	p^{-1} \mathbb{E} h'(\bm{H},\eta_{\bm{H}} \cap \mathbb{H}(\bm{H},\bm{G},\alpha)^c) \mathbf 1\{\eta \cap \mathbb{H}(\bm{H},\bm{G},\alpha)=\emptyset\}\\
		&=p^{-1} \mathbb{E} h'(\bm{H},\eta_{\bm{H}} ) \mathbf 1\{\eta \cap \mathbb{H}(\bm{H},\bm{G},\alpha)=\emptyset\}\\
		&\le p^{-1} \mathbb{E} h'(\bm{H},\eta_{\bm{H}} ) .
	\end{align*}
	Since $p^{-1}\le \sqrt 2$ if $\frac{\max(r(\bm{H}),r(\bm{H}_I,\bm{G}))}{\|z(\bm{H})-z(\bm{H}_I,\bm{G})\|}$ is small enough, the proof of the lemma is complete.
\end{proof}

\section{Proof of Theorem \ref{thmsigma}}

As a last intermediate result before the proof of Theorem \ref{thmsigma}, we show in the following lemma how to bound the expected number of pairs of large cells whose centers are within a distance less than or equal to some $D>0$. The lemma will be used together with \eqref{Ken}  and Lemma \ref{Le:le1} to exclude local clustering of large cells. It can be understood as a counterpart to Lemma \ref{Le:E2bou} that excludes asymptotic long-range dependencies.

\begin{lemma}\label{Le:exc}
	Let $D>0$, $W \subset \mathbb{R}^d$ be compact, $I \subsetneq \{1,\dots,d+1\}$, $a \in (0,1)$, $R>0$ and $y>0$. We have
	\begin{align*}
		&\frac{\gamma^{d+1}}{(d+1)!}\int_{\H^{d+1}} \int_{\H^{d+1-|I|}} \mathbf{1}\{z(\bm{H}) \in W\}\,\mathbf{1}\{ \|z(\bm{H})-z(\bm{H}_I,\bm{G})\|\le D\}  \nonumber\\
		&\,\,  \times \mathbb{P}(\eta \cap \H_{B(\bm{H}) \cup B(\bm{H}_I,\bm{G})}=\emptyset ,\,\Sigma(C(\bm{H},\eta_{\bm{H},\bm{G}}))>y,\,\Sigma(C((\bm{H}_I,\bm{G}),\eta_{\bm{H},\bm{G}}))>y)\nonumber\\
		&\,\, \times \mathbf{1}\{\min\{r(\bm{H}),r(\bm{H}_I,\bm{G})\}\le aR\}\mathbf{1}\{\max\{r(\bm{H}),r(\bm{H}_I,\bm{G})\}\le R\} \,\mu_{d-1}^{d+1-|I|}(\mathrm{d}\bm{G})\,\mu_{d-1}^{d+1}(\mathrm{d}\bm{H})\nonumber\\
		&\le 	2^{d+2-|I|} \gamma^{(d)} (D+R)^{d+1-|I|} \lambda_d(W) \mathbb{P}\Big(\Sigma(Z)>y,\vt(Z) \ge \frac{\tau y^{1/k}-2aR}{\tau y^{1/k}+2aR} \Big),
	\end{align*}
	where for $I=\{i_1,\dots,i_m\}$ we set $\bm{H}_I=(H_{i_1},\dots,H_{i_m})$ and $\tau$ is the constant from \eqref{iso}.
\end{lemma}

\begin{proof}	
	We assume that $r(\bm{H}_I,\bm{G}) \ge r(\bm{H})$ (at the cost of a factor 2). By the triangle inequality, $ \|z(\bm{H})-z(\bm{H}_I,\bm{G})\|\le D$ and $r(\bm{H}_I,\bm{G})\le R$ imply that $d(G_i,z(\bm{H}))\le D+R$ for $i \in \{1,\dots,d+1\}\setminus I$. Using that $\mu_{d-1}(\mathbb{H}_{B(z(\bm{H}),D+R)}) \le 2(D+R)$ and that $\Sigma(C(\bm{H},\eta_{\bm{H},\bm{G}})) \le \Sigma(C(\bm{H},\eta_{\bm{H}}))$, we obtain for the left-hand side in the statement of the lemma
	\begin{align*}
		\frac{2^{d+2-|I|}  (D+R)^{d+1-|I|}\gamma^{d+1}}{(d+1)!}\int_{\H^{d+1}} & \mathbf{1}\{z(\bm{H}) \in W\} \mathbb{P}(\eta \cap \H_{B(\bm{H})}=\emptyset,\,\Sigma(C(\bm{H},\eta_{\bm{H}}))>y)\\
		&\times\,\mathbf{1}\{r(\bm{H})\le aR\}\,\mu_{d-1}^{d+1}(\mathrm{d}\bm{H}).
	\end{align*}
	Now we use that the conditions $r(\bm{H})\le aR$ and $\Sigma(\bm{H},\eta)>y$ imply by the definition of $\Phi$ (see \eqref{defphi}) and the isoperimetric inequality \eqref{iso} that
	\begin{align*}
		\min \left\{\frac{R}{r}:\, rB^d+z \subset K\subset RB^d+z,\,r,R>0,\,z \in \mathbb{R}^d\right\} \ge \frac{\Phi(K)}{2r(K)} \ge \frac{\tau \Sigma(K)^{1/k}}{2r(K)},\quad K \in \mathcal{K}^d,
	\end{align*}
	where $r(K)$ is the inradius of $K$. Hence, for $\Sigma(C(\bm{H},\eta_{\bm{H}}))>y$ and $r(\bm{H})\le aR$ we have that $\vt(C(\bm{H},\eta_{\bm{H}}))\ge \frac{\tau y^{1/k}-2aR}{\tau y^{1/k}+2aR}$, where $\vt$ is the deviation function given at \eqref{defvt}. This yields the bound
	\begin{align*}
		&\frac{2^{d+2-|I|}  (D+R)^{d+1-|I|}\gamma^{d+1}}{(d+1)!}\int_{\H^{d+1}} \mathbf{1}\{z(\bm{H}) \in W\}\\
		&\qquad \quad \times\,  \mathbb{P}\Big(\eta \cap \H_{B(\bm{H})}=\emptyset,\,\Sigma(C(\bm{H},\eta_{\bm{H}}))>y,\vt(C(\bm{H},\eta_{\bm{H}}))\ge \frac{\tau y^{1/k}-2aR}{\tau y^{1/k}+2aR}\Big)\,\mu_{d-1}^{d+1}(\mathrm{d}\bm{H}).
	\end{align*}
	From here the assertion follows from the definition of the typical cell $Z$ (see \eqref{eqn:typZ}).
\end{proof}

\begin{proof}[Proof of Theorem \ref{thmsigma}]
	Let $W \subset \mathbb{R}^d$ be compact and $c>0$. Using \eqref{parZ} we find for all Borel sets $A \subset W$, $y \ge c$ and $nc \ge G^{-1}(\gamma^{-1/k})$ that
	\begin{align}
		&\E \xi_{n}(A \times (y,\infty))=n \gamma^{(d)} \lambda_d(A)\P(n^{-1}G(\Sigma(Z))>y) = \gamma^{(d)} \lambda_d(A) y^{-1},\label{intmeasgen}
	\end{align}
	where $Z$ is the typical cell in $\textsf{m}$. Hence, for $nc \ge G^{-1}(\gamma^{-1/k})$ the intensity measure of the restricted process $\xi_n \cap (W \times (c,\infty))$ is given by $\gamma^{(d)} \lambda_d(\cdot \cap W) \otimes \psi(\cdot \cap (c,\infty))$, where $\psi((a,b))=a^{-1}-b^{-1}$ for $0<a<b<\infty$ as defined in Section 2.
	
	We apply Theorem \ref{PPA} with
	\begin{align*}
		&g(\bm{H},\omega)=\mathbf{1}\{n^{-1/d}z(\bm{H}) \in W\} \mathbf{1}\{(\omega-\delta_{\bm{H}}) \cap \mathbb{H}_{B(\bm{H})}=\emptyset\}\mathbf{1}\{n^{-1}G(\Sigma(C(\bm{H},\omega))) >c\},\\
		&f(\bm{H},\omega)=(n^{-1/d}z(\bm{H}),n^{-1}G(\Sigma(C(\bm{H},\omega)))),\quad \omega \in \mathbf{N}_{\mathbb{H}},\, \bm{H} \in \omega^{(d+1)},
	\end{align*}
	the stopping set $\mathcal{S}$ from \eqref{defS} and $S_{{\bm{H}}}:=\mathbb{H}_{B(z(\bm{H}),(\cos \alpha)^{-2}r(\bm{H}))}$, where the value of $\alpha \in (0,\frac {\pi}{6})$ will be specified later. Note that \eqref{fgass} holds by \eqref{relCS}. Let $\nu$ be a Poisson process on $\R^d \times (0,\infty)$ with intensity measure $\gamma^{(d)} \lambda_d \otimes \psi$. Using Remark \ref{rem:Poi}(a), we find
	\begin{align*}
		\mathbf{d_{KR}}(\xi_n \cap (W \times (c,\infty)),\nu \cap (W \times (c,\infty)))\le E_1+2E_2+E_4+E_5+E_6
	\end{align*}
	with
	\begin{align*}
		E_1&= \frac{2\gamma^{d+1}}{(d+1)!} \int_{\mathbb{H}^{d+1}} \E g(\bm{H},\eta_{\bm{H}}) \mathbf{1}\{\mathcal{S}(\bm{H},\eta_{\bm{H}}) \not \subset S_{\bm{H}}\}\,\mu_{d-1}^{d+1} (\mathrm{d}\bm{H}),\\
		E_2&=\frac{2\gamma^{2d+2}}{((d+1)!)^2} \int_{\mathbb{H}^{d+1}} \int_{\mathbb{H}^{d+1}} \E \mathbf{1}\{\eta_{(\bm{H},\bm{G})} \cap S_{\bm{H}} \cap S_{\bm{G}} \neq \emptyset\} \tilde{g}(\bm{H},\eta_{\bm{H}}) \E \tilde{g}(\bm{G},\eta_{\bm{G}}) \nonumber\\
		& \quad \quad \quad \quad \times \mu_{d-1}^{d+1}(\mathrm{d}\bm{G}) \mu_{d-1}^{d+1}(\mathrm{d}\bm{H}),\\
		E_4&=\frac{2\gamma^{2d+2}}{((d+1)!)^2} \int_{\mathbb{H}^{d+1}} \int_{\mathbb{H}^{d+1}} \E \mathbf{1}\{\eta_{(\bm{H},\bm{G})}\cap S_{\bm{H}} \cap S_{\bm{G}} \neq \emptyset\} \tilde{g}(\bm{H},\eta_{(\bm{H},\bm{G})}) \tilde{g}(\bm{G},\eta_{(\bm{H},\bm{G})})\nonumber\\
		& \quad \quad \quad \quad \times \mathbf{1}\{r(\bm{H})+r(\bm{G})\le\|z(\bm{H})-z(\bm{G})\|\}\, \mu_{d-1}^{d+1}(\mathrm{d}\bm{G}) \mu_{d-1}^{d+1}(\mathrm{d}\bm{H}),\\
		E_5&=\frac{2\gamma^{2d+2}}{((d+1)!)^2} \int_{\mathbb{H}^{d+1}} \int_{\mathbb{H}^{d+1}}  \E \mathbf{1}\{\eta_{(\bm{H},\bm{G})} \cap S_{\bm{H}} \cap S_{\bm{G}} = \emptyset\} \tilde{g}(\bm{H},\eta_{(\bm{H},\bm{G})}) \tilde{g}(\bm{G},\eta_{(\bm{H},\bm{G})})\\
		& \quad \quad \quad \quad \times \, \P(\eta\cap S_{\bm{H}} \cap S_{\bm{G}} \neq \emptyset)\mathbf{1}\{r(\bm{H})+r(\bm{G})\le\|z(\bm{H})-z(\bm{G})\|\} \,\mu_{d-1}^{d+1}(\mathrm{d}\bm{G}) \mu_{d-1}^{d+1}(\mathrm{d}\bm{H}),\\
		E_6&=\frac{2\gamma^{d+1}}{(d+1)!} \sum_{\emptyset \subsetneq I \subsetneq \{1,\dots,d+1\}} \frac{\gamma^{d+1-|I|}}{(d+1-|I|)!} \int_{\mathbb{H}^{d+1}} \int_{\mathbb{H}^{d+1-|I|}} \E  \tilde{g}(\bm{H},\eta_{(\bm{H},\bm{G})}) \tilde{g}(\bm{G},\eta_{(\bm{H},\bm{G})})\nonumber\\
		& \quad \quad \quad \quad \times \mathbf{1}\{r(\bm{H})+r(\bm{H}_I,\bm{G})\le\|z(\bm{H})-z(\bm{H}_I,\bm{G})\|\}\, \mu_{d-1}^{d+1-|I|}(\mathrm{d}\bm{G}) \mu_{d-1}^{d+1}(\mathrm{d}\bm{H}),
	\end{align*}
	where for $I=\{i_1,\dots,i_m\}$ we set $\bm{H}_I=(H_{i_1},\dots,H_{i_m})$, $(\bm{H}_I,\bm{G})=(H_{i_1},\dots,H_{i_m},\bm{G})$ and write $\tilde{g}(\bm{H},\omega):=g(\bm{H},\omega) \mathbf{1}\{\mathcal{S}(\bm{H},\omega) \subset S_{\bm{H}}\}$. Now we bound $E_1, E_2, E_4, E_5, E_6$ separately.
	
	{\em The estimate of $E_1$.} Since $\mathcal{S}(\bm{H},\omega)=\mathcal{S}(\bm{H},\omega \cap \H^{B(\bm{H})})$, we find from independence of the processes $\eta \cap \H_{B(\bm{H})}$ and $\eta \cap \H^{B(\bm{H})}$ that $E_1$ is given by
	\begin{align*}
		\frac{2\gamma^{d+1}}{(d+1)!}&\int_{\H^{d+1}} \P(\mathcal{S}(\bm{H},\eta) \not \subset S_{\bm{H}}, n^{-1}G(\Sigma(C(\bm{H},\eta_{\bm{H}})))>c)\,
		\mathrm{e}^{-2\gamma r(\bm{H})}\, \mathbf{1}\{z(\bm{H})\in W_n\} \\
		&\quad \times\,\mu_{d-1}^{d+1}(\mathrm{d}\bm{H}).
	\end{align*}
	Note that $\mathcal{S}(\bm{H},\eta) \not \subset S_{\bm{H}}$ implies by the definition of the stopping set $\mathcal{S}$ (see \eqref{defS}) that $\max_{i \in I} R_i(z(\bm{H}),\eta \cap \H^{B(\bm{H})})>(\cos \alpha)^{-1}r(\bm{H})$. This yields that $\vt(C(\bm{H},\eta_{\bm{H}}))>\frac {1-\cos \alpha}{1+\cos \alpha}$, where $\vt$ is given in \eqref{defvt}. Hence, we obtain that $E_1$ is bounded from above by
	\begin{align}
		&n\gamma^{(d)} \lambda_d(W)  \P(n^{-1}G(\Sigma(Z))>c) \P\Big(\vt(Z)> \frac {1-\cos \alpha}{1+\cos \alpha} \Big\vert \, n^{-1}G(\Sigma(Z))>c\Big)\nonumber\\
		& \quad\le c_1 \exp\left\{-c_0s\Big(\frac {1-\cos \alpha}{1+\cos \alpha}\Big)G^{-1}(nc)^{1/k}\gamma\right\},\label{E1b}
	\end{align}
	where $Z$ is the typical cell with distribution given in \eqref{Ken} and the inequality follows from \eqref{parZ} and \eqref{Ken}. Since $\tau \gamma G^{-1}(nc)^{1/k} \sim \log n$ by \eqref{asyyn} and $s(\varepsilon)>0$ for $\varepsilon>0$, $E_1$ is bounded by $c_1 n^{-b}$ for some $b>0$.

	{\em The estimate of $E_2$.} We split the event $\{\eta_{(\bm{H},\bm{G})} \cap S_{\bm{H}} \cap S_{\bm{G}} \neq \emptyset\}$ from the first indicator in $E_2$ into $\{\delta_{(\bm{H},\bm{G})} \cap S_{\bm{H}} \cap S_{\bm{G}} \neq \emptyset\}$ and $\{\eta \cap S_{\bm{H}} \cap S_{\bm{G}} \neq \emptyset\}$. This allows us to bound $E_2$ by
	\begin{align}
		&\frac{2\gamma^{2d+2}}{((d+1)!)^2}\int_{\H^{d+1}} \int_{\H^{d+1}} \mathbf{1}\{\delta_{(\bm{H},\bm{G})} \cap S_{\bm{H}} \cap S_{\bm{G}} \neq \emptyset\} \E \tilde g(\bm{H},\eta_{\bm{H}}) \E \tilde g(\bm{G},\eta_{\bm{G}})  \nonumber\\ 
		&\qquad \qquad \qquad \qquad  \qquad\times\,\mu_{d-1}^{d+1}(\mathrm{d}\bm{G})\,\mu_{d-1}^{d+1}(\mathrm{d}\bm{H})\label{E21}\\
		&\,+\frac{2\gamma^{2d+2}}{((d+1)!)^2}\int_{\H^{d+1}}\int_{\H^{d+1}} \E \mathbf{1}\{\eta \cap S_{\bm{H}} \cap S_{\bm{G}} \neq \emptyset\} \tilde g(\bm{H},\eta_{\bm{H}})\, \E \tilde g(\bm{G},\eta_{\bm{G}})\nonumber\\ 
		&\qquad \qquad \qquad \qquad  \qquad \times\,\mu_{d-1}^{d+1}(\mathrm{d}\bm{G})\,\mu_{d-1}^{d+1}(\mathrm{d}\bm{H}).\label{E22}
	\end{align}
	Analogously to the bound of $E_2$ in the proof of Theorem \ref{thmr}, we find from Lemma \ref{bplem} with $\ell=1$ (note that $\nabla_1(v)=1$ for all $v \in S^{d-1}$) that \eqref{E21} is bounded by
	\begin{align}
		&\frac{2^{d+2} \gamma^{2d+2} }{(d+1)! (d+1)} \int_{\H^{d+1}} \int_{\H_{B(\bm{H})}} \int_{S_{G_1^\perp}} \int_{(S^{d-1})^{d}} \int_{0}^\infty  \int_{W_n\cap(G_1-su_1) } \E \tilde g(\bm{H},\eta_{\bm{H}}) \E \tilde g(\bm{G},\eta_{\bm{G}})  \nonumber\\
		&\quad\times\,  \Delta_d(u_{1:d+1}) \mathbf{1}_\mathsf{P}(u_{1:d+1})\lambda_{d-1}(\mathrm{d}w)\, \mathrm{d}s \,\sigma_{d-1}^{d}(\mathrm{d}u_{2:d+1})\,\sigma_0(\mathrm{d} u_1)\,\mu_{d-1}(\mathrm{d}G_1)\,\mu_{d-1}^{d+1}(\mathrm{d}\bm{H}),\label{E22lem6}
	\end{align}
	where $\bm{G}:=(G_1,H(u_2,\langle u_2,w \rangle +s),\dots,H(u_{d+1},\langle u_{d+1},w \rangle +s))$. Since $\mu_{d-1}(\mathbb{H}_{B(\bm{H})})= 2r(\bm{H})$ (see \eqref{repmu}), we have
	\begin{align*}
		\int_{\mathbb{H}_{B(\bm{H})}} \int_{S_{G_1^\perp}} \mathbf 1\{u \in \cdot\} \,\sigma_0(\mathrm{d}u)\,\mu_{d-1}(\mathrm{d}G_1)=2r(\bm{H})\sigma_{d-1}(\cdot).
	\end{align*}
	Hence, \eqref{E22lem6} is given by
	\begin{align*}
		\frac{2^{d+3}\gamma^{2d+2}}{(d+1)! (d+1)} &\int_{\H^{d+1}}  \int_{(S^{d-1})^{d+1}} \int_{0}^\infty  \int_{W_n\cap(G_1-s u_1) }r(\bm{H}) \, \E \tilde g(\bm{H},\eta_{\bm{H}}) \E \tilde g(\bm{G},\eta_{\bm{G}})  \nonumber\\
		&\quad\times\,  \Delta_d(u_{1:d+1}) \mathbf{1}_\mathsf{P}(u_{1:d+1})\lambda_{d-1}(\mathrm{d}w)\, \mathrm{d}s \,\sigma_{d-1}^{d+1}(\mathrm{d}u_{1:d+1}) \mu_{d-1}^{d+1}(\mathrm{d}\bm{H}).
	\end{align*}
	Using that $\lambda_{d-1}(W_n\cap(G_1-su_1)) \le n^{(d-1)/d} \text{diam}(W)^{d-1} $, we obtain by definition of the typical cell $Z$ the bound
	\begin{align}
		c_2n^{\frac{d-1}{d}} \P(n^{-1}G(\Sigma(Z))>c) \int_{\H^{d+1}} r(\bm{H}) \E \tilde g(\bm{H},\eta_{\bm{H}})  \, \mathrm{e}^{-2\gamma r(\bm{H})} \mu_{d-1}^{d+1}(\mathrm{d}\bm{H}). \label{eqn:E2halvtyp}
	\end{align}
	Next we distinguish by the value of $r(\bm{H})$. If $r(\bm{H})>\frac{\log n}{\gamma}$, the integral in \eqref{eqn:E2halvtyp} is by \cite[Theorem 7.3.2]{schneider2008stochastic} of order $n^{-1}$. If $r(\bm{H})\le\frac{\log n}{\gamma}$, it is by \eqref{intmeasgen} of order $\log n$. Hence, \eqref{E21} is bounded by $c_3 n^{-1/d} \log n$. 
	
	To bound \eqref{E22}, we consider the first probability in the integrand and find by the Mecke equation (see \cite[Theorem 4.1]{last2017lectures}) that for fixed $\bm{H} \in \mathbb{H}^{d+1}$,
	\begin{align}
		\E \mathbf{1}\{\eta \cap S_{\bm{H}} \cap S_{\bm{G}} \neq \emptyset\} \tilde g(\bm{H},\eta_{\bm{H}})&\le \E \sum_{H \in \eta} \mathbf{1}\{H \in S_{\bm{H}} \cap S_{\bm{G}}\} \tilde g(\bm{H},\eta_{\bm{H}}) \nonumber\\
		&= \gamma\int_{S_{\bm{H}} \cap S_{\bm{G}}} \E \tilde g(\bm{H},\eta_{\bm{H}}+\delta_H)\,\mu_{d-1}(\mathrm{d}H). \label{meckebou}
	\end{align}
	Since $g(\bm{H},\omega_1)\ge g(\bm{H},\omega_2)\ge \tilde g(\bm{H},\omega_2)$ for $\omega_1 \subset \omega_2$, we have $g(\bm{H},\eta_{\bm{H}}) \ge \tilde g(\bm{H},\eta_{\bm{H}}+\delta_H)$ a.s. Hence, \eqref{meckebou} is bounded by
	\begin{align*}
		\gamma \mu_{d-1}(S_{\bm{H}} \cap S_{\bm{G}})\, \E g(\bm{H},\eta_{\bm{H}}).
	\end{align*}
	From \eqref{mubou} we find that $\mu_{d-1}(S_{\bm{H}} \cap S_{\bm{G}}) \le \frac{4 (\cos \alpha)^{-4}r(\bm{H}) r(\bm{G}) \omega_{d-1}}{\omega_d \|z(\bm{H})-z(\bm{G})\|}$. Similarly to the bound of $E_5$ in the proof of Theorem \ref{thmr}, we conclude that \eqref{E22} (and hence also $E_2$) is bounded by $c_4 n^{-1/d} (\log n)^2$.
	
	{\em The estimate of $E_4$.} Fix $\delta>0$ and let $r_n:=\frac{\log n}{2\gamma}$. We partition the integration area of $E_4$ into the Borel sets
	\begin{align*}
		I_1&=\{(\bm{H},\bm{G})\in \mathbb{H}^{2d+2}:r(\bm{H}) \vee r(\bm{G})>(2+\delta) r_n\},\\
		I_2&=\{(\bm{H},\bm{G})\in \mathbb{H}^{2d+2}: r(\bm{H}) \wedge r(\bm{G})>\frac{r_n}{2},r(\bm{H}) \vee r(\bm{G})\le (2+\delta)r_n,\\
		&\qquad\|z(\bm{H})-z(\bm{G})\|\le (\log n)^2\},\\
		I_3&=\{(\bm{H},\bm{G})\in \mathbb{H}^{2d+2}:r(\bm{H}) \wedge r(\bm{G})\le \frac{r_n}{2},r(\bm{H}) \vee r(\bm{G})\le (2+\delta) r_n,\\
		&\qquad\|z(\bm{H})-z(\bm{G})\|\le (\log n)^2\},\\
		I_4&=\{(\bm{H},\bm{G})\in \mathbb{H}^{2d+2}:r(\bm{H}) \vee r(\bm{G})\le (2+\delta)r_n, \|z(\bm{H})-z(\bm{G})\|> (\log n)^2\},
	\end{align*}
	where $a \vee b:=\max(a,b)$ and $a \wedge b  :=\min(a,b)$ for $a,b \in \mathbb{R}$. This yields for $E_4$ the bound
	\begin{align*}
		&	\frac{2\gamma^{2d+2}}{((d+1)!)^2}\Big(\int_{I_1} \mathbf 1\{z(\bm{H}), z(\bm{G})\in W_n\}\,e^{-\gamma \mu_{d-1}(\mathbb{H}_{B(\bm{H})}\cup \mathbb{H}_{B(\bm{G})})} \,\mu_{d-1}^{2d+2}(\mathrm{d}(\bm{H},\bm{G}))\\
		&+ \int_{I_2} \mathbf 1\{z(\bm{H}), z(\bm{G})\in W_n\} \mathbf 1\{r(\bm{H})+r(\bm{G})\le\|z(\bm{H})-z(\bm{G})\|\} \, \mathrm{e}^{-\gamma \mu_{d-1}(\mathbb{H}_{B(\bm{H})}\cup \mathbb{H}_{B(\bm{G})})}\\ 
		&\qquad \times\mu_{d-1}^{2d+2}(\mathrm{d}(\bm{H},\bm{G}))\\
		&+ \int_{I_3} \mathbb{E} \tilde{g}(\bm{H},\eta_{\bm{H},\bm{G}})\,\tilde{g}(\bm{G},\eta_{\bm{H},\bm{G}})\,\mu_{d-1}^{2d+2}(\mathrm{d}(\bm{H},\bm{G}))\\
		& +\int_{I_4}  \mathbb{E} \mathbf{1}\{\eta_{(\bm{H},\bm{G})}\cap S_{\bm{H}} \cap S_{\bm{G}} \neq \emptyset\} \tilde {g}(\bm{H},\eta_{\bm{H},\bm{G}})\,\tilde {g}(\bm{G},\eta_{\bm{H},\bm{G}})\,\mu_{d-1}^{2d+2}(\mathrm{d}(\bm{H},\bm{G}))\Big).
	\end{align*}
	From Lemma \ref{Le:le1}(a) (with $R:=\frac{(2+\delta)\log n}{2 \gamma}$) we find that the integral over $I_1$ is bounded by $c_4 n^{-\delta} \log n$. By Lemma \ref{Le:le1}(b) (with $a:=(2+\delta)^{-1} (2-L(\frac{1}{4+2\delta}))^{-1/2}$ and $D:=(\log n)^2$) we have that the integral over $I_2$ is bounded by $c_5n^{1-(2+\delta)a(2-L(a))} (\log n)^{2d}$. To see that the exponent of $n$ is negative, observe that the fact that $L(\cdot)\in (0,1)$ yields $a> \frac{1}{4+2\delta}$. Since $L$ is decreasing, this implies that
	\begin{align*}
		(2+\delta)a(2-L(a))=\frac{2-L(a)}{\sqrt{2-L(\frac{1}{4+2\delta})}}\ge\sqrt{2-L(a)}>1.
	\end{align*}
	The integral over $I_3$ is by Lemma \ref{Le:exc} (with the same choices of $a$, $D$ and $R$ as above) bounded by
	\begin{align*}
		&c_6 n (\log n)^{2d+2} \mathbb{P}\Bigg(n^{-1}G(\Sigma(Z))>c,\,\vt(Z)>\frac{\sqrt{2-L(\frac{1}{4+2\delta})}\tau \gamma G^{-1}(nc)^{1/k}- \log n}{\sqrt{2-L(\frac{1}{4+2\delta})}\tau \gamma G^{-1}(nc)^{1/k}+ \log n}\Bigg)\\
		&\quad \le c_7 (\log n)^{2d+2}  \mathbb{P}\Bigg(\vt(Z)>\frac{\sqrt{2-L(\frac{1}{4+2\delta})}\tau \gamma G^{-1}(nc)^{1/k}- \log n}{\sqrt{2-L(\frac{1}{4+2\delta})}\tau \gamma G^{-1}(nc)^{1/k}+ \log n}\, \Bigg| \, n^{-1}G(\Sigma(Z))>c\Bigg).
	\end{align*}
	Since $\sqrt{2-L(\frac{1}{4+2\delta})}>1$ and $\tau \gamma G^{-1}(nc)^{1/k} \sim \log n$ by \eqref{fgass}, it follows from \eqref{parZ} and \eqref{Ken} that the integral over $I_3$ is bounded by $c_7 n^{-b}$ for some $b>0$. We split the integral over $I_4$ into
	\begin{align}
		&\int_{I_4}  \mathbf{1}\{\delta_{(\bm{H},\bm{G})}\cap S_{\bm{H}} \cap S_{\bm{G}} \neq \emptyset\} \mathbb{E} \tilde {g}(\bm{H},\eta_{\bm{H},\bm{G}})\,\tilde {g}(\bm{G},\eta_{\bm{H},\bm{G}})\,\mu_{d-1}^{2d+2}(\mathrm{d}(\bm{H},\bm{G}))\label{eqn:E4I4a}\\
		&\quad +\int_{I_4}  \mathbb{E} \mathbf{1}\{\eta\cap S_{\bm{H}} \cap S_{\bm{G}} \neq \emptyset\} \tilde {g}(\bm{H},\eta_{\bm{H},\bm{G}})\,\tilde {g}(\bm{G},\eta_{\bm{H},\bm{G}})\,\mu_{d-1}^{2d+2}(\mathrm{d}(\bm{H},\bm{G})).\label{eqn:E4I4b}
	\end{align}
	For \eqref{eqn:E4I4a} we obtain from Lemma \ref{Le:E2bou} with $\mu=0$ (where we use that $\frac{\max(r(\bm{H}),r(\bm{H}_I,\bm{G}))}{\|z(\bm{H})-z(\bm{H}_I,\bm{G})\|}$ becomes on $I_4$ arbitrarily small for $n$ large enough) the bound
	\begin{align*}
		\int_{I_4}  \mathbf{1}\{\delta_{(\bm{H},\bm{G})}\cap S_{\bm{H}} \cap S_{\bm{G}} \neq \emptyset\} \mathbb{E} 
		{g}(\bm{H},\eta_{\bm{H}} )\,\mathbb{E} {g}(\bm{G},\eta_{\bm{G}})\,\mu_{d-1}^{2d+2}(\mathrm{d}(\bm{H},\bm{G})).
	\end{align*}
	From here we can proceed as in the bound of \eqref{E21} and find that \eqref{eqn:E4I4a} is bounded by $c_8 n^{-1/d} \log n$. 
	
	Now we discuss \eqref{eqn:E4I4b}. 
	First notice that \eqref{eqn:E4I4b} is by \eqref{meckebou} (with $\tilde {g}(\bm{H},\eta_{\bm{H},\bm{G}})\,\tilde {g}(\bm{G},\eta_{\bm{H},\bm{G}})$ instead of $\tilde {g} (\bm{H},\eta_{\bm{H}})$) bounded by
	\begin{align*}
		\gamma \int_{I_4}  \int_{S_{\bm{H}} \cap S_{\bm{G}}}\mathbb{E}  \tilde {g}(\bm{H},\eta_{\bm{H},\bm{G}}+\delta_H)\,\tilde {g}(\bm{G},\eta_{\bm{H},\bm{G}}+\delta_H)\,\mu_{d-1}(\mathrm{d}H)\,\mu_{d-1}^{2d+2}(\mathrm{d}(\bm{H},\bm{G})).
	\end{align*}
	Here we use Lemma \ref{Le:E2bou} with $\mu:=\delta_H$ and obtain the bound
	\begin{align}
		2\gamma\int_{I_4} \mu_{d-1}(S_{\bm{H}} \cap S_{\bm{G}})\mathbb{E}g(\bm{H},\eta_{\bm{H}}) \mathbb{E}g(\bm{G},\eta_{\bm{G}})\, \mu_{d-1}^{2d+2}(\mathrm{d}(\bm{H},\bm{G})).\label{eqn:E4I4bref}
	\end{align}
	Now we invoke the bound $\mu_{d-1}(S_{\bm{H}} \cap S_{\bm{G}}) \le \frac{4 (\cos \alpha)^{-4}r(\bm{H}) r(\bm{G}) \omega_{d-1}}{\omega_d \|z(\bm{H})-z(\bm{G})\|}$ from \eqref{mubou}. Introducing spherical coordinates, we find analogously to the bound of \eqref{E22} that \eqref{eqn:E4I4b} is bounded by $c_9n^{-1/d}(\log n)^2$.
	
	{\em The estimate of $E_5$.} We partition the integration area of $E_5$ into the same sets as for $E_4$. This gives the bound 
	\begin{align*}
		&	\frac{2\gamma^{2d+2}}{((d+1)!)^2}\Big(\int_{I_1} \mathbf 1\{z(\bm{H}), z(\bm{G})\in W_n\}\,e^{-\gamma \mu_{d-1}(\mathbb{H}_{B(\bm{H})}\cup \mathbb{H}_{B(\bm{G})})} \,\mu_{d-1}^{2d+2}(\mathrm{d}(\bm{H},\bm{G}))\\
		&+ \int_{I_2} \mathbf 1\{z(\bm{H}), z(\bm{G})\in W_n\} \mathbf 1\{r(\bm{H})+r(\bm{G})\le\|z(\bm{H})-z(\bm{G})\|\} \,\mathrm{e}^{-\gamma \mu_{d-1}(\mathbb{H}_{B(\bm{H})}\cup \mathbb{H}_{B(\bm{G})})}\\ 
		&\qquad \times \mu_{d-1}^{2d+2}(\mathrm{d}(\bm{H},\bm{G}))\\
		&+ \int_{I_3} \mathbb{E} \tilde{g}(\bm{H},\eta_{\bm{H},\bm{G}})\,\tilde{g}(\bm{G},\eta_{\bm{H},\bm{G}})\,\mu_{d-1}^{2d+2}(\mathrm{d}(\bm{H},\bm{G}))\\
		&+\int_{I_4} \P(\eta\cap S_{\bm{H}} \cap S_{\bm{G}} \neq \emptyset) \mathbb{E}  \tilde {g}(\bm{H},\eta_{\bm{H}})\,\tilde {g}(\bm{G},\eta_{\bm{G}})\,\mu_{d-1}^{2d+2}(\mathrm{d}(\bm{H},\bm{G}))\Big).
	\end{align*}
	The integrals over $I_1$, $I_2$ and $I_3$ can be bounded as the analogous terms in the estimate of $E_4$. For the integral over $I_4$ we note that $\P(\eta\cap S_{\bm{H}} \cap S_{\bm{G}} \neq \emptyset) \le \gamma \mu_{d-1}(S_{\bm{H}} \cap S_{\bm{G}})$ and can thus argue as for \eqref{eqn:E4I4bref}.
	
	{\em The estimate of $E_6$.} Analogously to the bounds of $E_4$ and $E_5$, we split the integration area of $E_6$ into the sets $I_1$, $I_2$, $I_3$ and $I_4$ (with $\bm{G}$ replaced by $(\bm{H}_I,\bm{G})$). As above, we apply Lemma \ref{Le:le1} and Lemma \ref{Le:exc} to bound the integrals over $I_1$, $I_2$ and $I_3$. For the integral over $I_4$ we use Lemma \ref{Le:E2bou} and find for $n$ large enough (by which we mean that $\frac{\max(r(\bm{H}),r(\bm{H}_I,\bm{G}))}{\|z(\bm{H})-z(\bm{H}_I,\bm{G})\|}\le \frac{(2+\delta)\log n}{2\gamma (\log n)^2}$  is small enough to apply Lemma \ref{Le:E2bou}) the bound
	\begin{align}
		&	\frac{4\gamma^{d+1}}{(d+1)!} \sum_{\emptyset \subsetneq I \subsetneq \{1,\dots,d+1\}} \frac{\gamma^{d+1-|I|}}{(d+1-|I|)!}  \int_{\mathbb{H}^{d+1}}\int_{\mathbb{H}^{d+1-|I|}}  \mathbb{E} {g}(\bm{H},\eta_{\bm{H}})  \mathbf 1\{R'(\bm{H},\eta)\le r(\bm{H})_\alpha\}  \nonumber\\
		&\quad \times \, \mathbb{E}g((\bm{H}_I,\bm{G}),\eta_{\bm{H}_I,\bm{G}})  \mathbf 1\{R'((\bm{H}_I,\bm{G}),\eta)\le r(\bm{H}_I,\bm{G})_\alpha\} \, \mu_{d-1}^{d+1-|I|}(\mathrm{d}\bm{G})\,\mu_{d-1}^{d+1}(\mathrm{d}\bm{H}),\label{E6s2}
	\end{align}
	where we recall the notation $r_\alpha:= \frac{r}{(\cos \alpha) (\cos 3\alpha)}$ for $r>0$ and $\alpha \in (0,\frac \pi 6)$ from Section 7. By homogeneity of $\Sigma$ and monotonicity under set inclusion,  $\Sigma(C(\bm{H},\eta))>G^{-1}(nc)$ implies that $\Sigma(B^d)\,R'(\bm{H},\eta)^k > G^{-1}(nc)$. Since $r(\bm{H})_\alpha \ge R'(\bm{H},\eta)$ and $\tau=2\Sigma(B^d)^{-1/k}$, this yields that $r(\bm{H})>\frac{\tau((\cos \alpha) (\cos 3\alpha))^{1/k}}{2} G^{-1}(ny)^{1/k}$. We apply the same argument to the terms with $(\bm{H}_I,\bm{G})$ instead of $\bm{H}$ and assume (at the cost of a factor 2) that $r(\bm{H})\ge r(\bm{H}_I,\bm{G})$. Thus we arrive for \eqref{E6s2} at the bound
	\begin{align}
		&\frac{8\gamma^{d+1}}{(d+1)!} \sum_{\emptyset \subsetneq I \subsetneq \{1,\dots,d+1\}} \frac{\gamma^{d+1-|I|}}{(d+1-|I|)!}  \int_{\mathbb{H}^{d+1}}\int_{\mathbb{H}^{d+1-|I|}} \mathbf 1\{z(\bm{H})\in W_n\} \mathbf 1\{z(\bm{H}_I,\bm{G})\in W_n\}  \nonumber\\
		&\quad \times \,\mathbf 1\Big\{r(\bm{H})\ge r(\bm{H}_I,\bm{G})>\frac{\tau((\cos \alpha) (\cos 3\alpha))^{1/k}}{2} G^{-1}(nc)^{1/k}\Big\} \,\mathrm{e}^{-2 \gamma r(\bm{H})-2\gamma r(\bm{H}_I,\bm{G})}  \nonumber\\ 	
		&\quad \times \,\mu_{d-1}^{d+1-|I|}(\mathrm{d}\bm{G})\,\mu_{d-1}^{d+1}(\mathrm{d}\bm{H}). \label{eqn:E6lem12}
	\end{align}
	From \eqref{asyyn} we have that $\tau \gamma G^{-1}(nc)^{1/k} \sim \log n$ as $n \to \infty$ for all $c>0$. Hence, we can bound \eqref{eqn:E6lem12} analogously to \eqref{rE6s11}. If $\alpha>0$ is small enough, we find that \eqref{eqn:E6lem12} (and hence $E_6$) is bounded by $c_{11} n^{-b}$ for some $b>0$. This finishes the proof of Theorem \ref{thmsigma}.
\end{proof}

 \end{document}